%% file: GVforOrbiCY3s.tex
\documentclass[10pt]{amsart}

\usepackage{amscd,amsmath,amssymb,euscript,tabu,multirow,float,tikz-cd,tensind,bm,mathtools,amsthm}
\tensordelimiter{?}
\usepackage[frame,cmtip,curve,arrow,matrix,line,graph]{xy}
\usepackage{xcolor,verbatim}
\usepackage[colorlinks=true, linkcolor=., citecolor=blue]{hyperref}
\usepackage{arydshln}
\usepackage{nicematrix}
\usepackage{caption}
\usepackage{enumerate}
\usepackage{stmaryrd}
\usepackage{tikz}
\usepackage{mathrsfs}
\usepackage{thmtools}
\usepackage{thm-restate}

\title[GV invariants for local orbifold CY3s of $A_{N}$-type]{Gopakumar-Vafa Invariants for Local Calabi-Yau Orbifolds of $\bm{A_{N}}$-type}
\author{Jim Bryan and Stephen Pietromonaco}

\date{}

\theoremstyle{plain}
\newtheorem{theorem}{Theorem}[section]

\newtheorem{conjecture}[theorem]{Conjecture}

\newtheorem{corollary}[theorem]{Corollary}
\newtheorem{proposition}[theorem]{Proposition}
\newtheorem{lemma}[theorem]{Lemma}

\theoremstyle{definition}
\newtheorem{definition}[theorem]{Definition}
\newtheorem{remark}[theorem]{Remark}
\newtheorem{example}[theorem]{Example}

\theoremstyle{plain}
\newtheorem{introtheorem}{Theorem}

\theoremstyle{definition}
\newtheorem{introdefinition}{Definition}

\newcommand{\Pic}{\operatorname{Pic}}

\newcommand{\num}{\operatorname{num}}

\newcommand{\dg}{\operatorname{deg}}

\newcommand{\Ext}{\operatorname{Ext}}
\newcommand{\Hom}{\operatorname{Hom}}

\newcommand{\Coef}{\operatorname{Coef}}
\newcommand{\SL}{\operatorname{SL}}
\newcommand{\Tot}{\operatorname{Tot}}

\newcommand{\Rep}{\operatorname{Rep}}
\newcommand{\Char}{\operatorname{char}}
\newcommand{\coker}{\operatorname{coker}}

\newcommand{\DD}{\operatorname{\mathbb{D}}}

\renewcommand{\natural}{\perp}

\renewcommand{\leq}{\leqslant}
\renewcommand{\geq}{\geqslant}

\newcommand\smvee{\raise0.3ex\hbox{$\scriptscriptstyle\vee$}}

\newcommand{\isom}{\cong}

\newcommand{\CC}{\mathbb C}
\newcommand{\Hilb}{\operatorname{Hilb}}

\newcommand{\QQ}{\mathbb Q}

\newcommand{\ZZ}{\mathbb Z}

\newcommand{\BB}{\mathcal B}
\newcommand{\OO}{\mathcal O}
\newcommand{\into}{\hookrightarrow}

\newcommand{\PP}{\mathbb{P}}

\newcommand{\F}{\mathsf{F}}
\newcommand{\XX}{\mathcal{X}}
\newcommand{\CCC}{\mathcal{C}}
\newcommand{\Sorb}{\mathcal{S}}
\newcommand{\vir}{\mathsf{vir}}

\newcommand{\mr}{\mathsf{mr}}

\newcommand{\nexc}{\textsf{n-exc}}

\newcommand{\ZZZ}{\mathsf{Z}}

\newcommand{\GW}{\mathsf{GW}}
\newcommand{\PT}{\mathsf{PT}}
\newcommand{\DT}{\mathsf{DT}}

\newcommand{\FM}{\mathsf{FM}}
\newcommand{\BS}{\mathsf{BS}}
\newcommand{\BM}{\mathsf{BM}}

\renewcommand{\bm}[1]{\mathbf{#1}}

\newcommand{\orb}{\mathsf{orb}}
\newcommand{\pt}{\mathsf{pt}}
\newcommand{\exc}{\mathsf{exc}}

\newcommand{\nrefined}{n^{\XX}_{g}(\beta; \bm{m})}
\newcommand{\Snrefined}{n^{\Sorb}_{g}(\beta; \bm{m})}

\newcommand{\LambdaWeightAN}{\Lambda_{A_{N}}^{\! \smvee}}

\begin{document}

\begin{abstract}
Let $\mathcal{X}$ be a local orbifold Calabi-Yau threefold whose
coarse space $X$ has transverse $A_{N}$ singularities along a smooth
non-compact curve. We define orbifold Gopakumar-Vafa invariants, and
we prove they are integers and satisfy a finiteness property. We
compute our invariants for local orbifold $K3$ surfaces, where we
prove an orbifold version of the classical Yau-Zaslow formula: we
show that for an orbifold $K3$ surface $\mathcal{S}$, the genus zero
Gopakumar-Vafa invariant of $\mathcal{S}$ in a class of square $2n$ is
the Euler characteristic of the Hilbert scheme of $n+1$ points on the
\emph{singular} surface $S$.

\end{abstract}

\maketitle


\section{Introduction}

\subsection{Ordinary Gopakumar-Vafa invariants}\label{subsec: intro ordinary GV section}

Let $Y$ be a Calabi-Yau threefold (CY3) by which we mean a smooth
quasi-projective variety over $\CC$ of dimension three with
$K_{Y}\cong \OO_{Y}$. One avenue to studying curve-like objects in $Y$
is through the stable pairs of Pandharipande-Thomas
\cite{Pandharipande-Thomas}. A stable pair in the bounded derived
category $D^{b}(Y)$ is a two-term complex, concentrated in degrees
$-1$ and $0$
\[
I^{\bullet} = [\OO_{Y} \xrightarrow{s} F] \in D^{b}(Y)
\] 
such that $F$ is a pure coherent sheaf with proper one-dimensional
support, and $\coker(s)$ is zero-dimensional. The discrete invariants
of a stable pair naturally live in the group $N_{\leq 1}(Y)$, which
denotes $K$-theory with proper support of dimension less than or equal
to one, up to numerical equivalence. There is a natural splitting
$N_{\leq 1}(Y)\cong N_{1}(Y)\oplus N_{0}(Y)$ where $N_{1}(Y) \subset
N_{\leq 1}(Y)$ is the subgroup of sheaves with Euler characteristic
zero, and where $N_{0}(Y)$ is the subgroup of zero-dimensional sheaves
which is isomorphic to $\ZZ$ via the Euler characteristic. A stable
pair then has class
\[
[F] = (\beta, n) \in N_{1}(Y) \oplus N_{0}(Y)
\]
where $\chi(F) = n$. There is a fine moduli scheme $P_{n}(Y, \beta)$
parameterizing stable pairs of class $(\beta, n)$, and we define the
Pandharipande-Thomas (PT) invariants $\PT_{\beta, n}(Y)$ as the
Behrend-weighted Euler characteristic of the moduli space
\cite{KaiBehrend}. These invariants are packaged into the PT partition
function
\[
\ZZZ_{Y}^{\PT}(q,y) =  \sum_{\beta \in N_{1}(Y)} \sum_{n \in \ZZ} \PT_{\beta, n}(Y) q^{\beta} y^{n}
\]
which is a priori a power series in multivariable $q$ whose coefficients are Laurent series in $y$. There is, however, a general structural symmetry induced by derived duality. Throughout this paper, we let $\QQ(y)^{y \leftrightarrow y^{-1}}$ denote the field of rational functions $h(y) \in \QQ(y)$ satisfying $h(y) = h(y^{-1})$. 
\begin{theorem}[{Bridgeland \cite{Bridgeland-PTDT}, Toda \cite{Toda-PTDT}}] \label{thm: Bridgeland and Toda Theorems}
For all $\beta \in N_{1}(Y)$, the formal series $\Coef_{q^{\beta}}\ZZZ_{Y}^{\PT}$ is the Laurent expansion in $y$ of an element $f_{\beta}(y) \in \QQ(y)^{y \leftrightarrow y^{-1}}$. 
\end{theorem}

In 1998 Gopakumar and Vafa \cite{Gopakumar-Vafa} defined via physics
integer invariants $n_{g}^{Y}(\beta)$ which give a virtual count of
curves in $Y$ of genus $g$ representing a non-zero curve class $\beta \in
N_{1}(Y)$. These are called the Gopakumar-Vafa (GV) invariants. The
following formula uniquely determines the GV invariants in terms of
the PT invariants 
\begin{equation}\label{eqn: Ordinary GV/PT formula}
\log \ZZZ_{Y}^{\PT}(q,y) = \sum_{\beta \neq 0} \sum_{k >0} \sum_{g\in \ZZ } n^{Y}_{g}(\beta) \psi_{-(-y)^{k}}^{g-1} \frac{q^{k \beta}}{k}
\end{equation}
where $\psi_{x} =(x^{\frac{1}{2}} + x^{-\frac{1}{2}})^{2}$. 

In physics, the GV invariants count BPS states of certain configurations of D-branes in string theory. In
mathematics, they underlie the PT invariants and have better enumerative
properties. It is inherent from the physics formulation that the GV
invariants have the following properties: they are all integers, they
are zero for $g<0$, and for fixed $\beta \in N_{1}(Y)$ we have
$n_{g}^{Y}(\beta) = 0$ for all but finitely many genera $g$. 

Treating (\ref{eqn: Ordinary GV/PT formula}) as a mathematical
definition of the GV invariants, the expected property that the genus
is non-negative becomes a highly non-trivial statement about the PT
invariants. Through the MNOP correspondence \cite{MNOP1}, recently
proved by Pardon \cite{Pardon}, this statement is equivalent to a
corresponding statement in Gromov-Witten theory proven in the late
2010s by Ionel-Parker and Doan-Ionel-Walpuski
\cite{Ionel-Parker-GVconj,Doan-Ionel-Walpuski} in symplectic geometry.

\subsection{Gopakumar-Vafa Invariants for Orbifolds}\label{subsection: GV for Orbifolds}

The goal of this paper is to extend the Gopakumar-Vafa story to
certain orbifold CY3s. Although the enumerative geometry of orbifolds
has been defined and studied for many years
\cite{Abramovich-Graber-Vistoli-2008,OrbiGW-Abramovich-Vistoli,Beentjes-Calabrese-Rennemo,
BryanCadmanYoung, BryanGraber}, no analog of (\ref{eqn: Ordinary GV/PT
formula}) has yet been proposed in general (see
\cite{BryanPietromonaco2024} for the case of CY3s with an
involution). We will define a version of the GV invariants for
orbifolds of the following form. Throughout the paper,
$\mu_{\scriptscriptstyle N+1} \subset \CC^{*}$ denotes the cyclic
group of $(N+1)$-st roots of unity, and
$K_{\XX }=\det(\Omega^1_{\XX })$ denotes the canonical
line bundle of a smooth orbifold $\XX $.

\begin{definition}\label{defn: new local orbifold CY3 defn}
An orbifold $\XX$ is called an \emph{orbifold CY3 of transverse
$A_{N}$-type} if it has $K_{\XX}\cong \OO_{\XX}$ and is locally
modelled on $[\CC^{2}/\mu_{N+1}]\times \CC$. In particular, the coarse
space $X$ has transverse $A_{N}$ singularities along a smooth curve.

We say $\XX $ is a \emph{local orbifold curve} or a \emph{local
orbifold surface of $A_{N}$-type} if it is of the following form:
\begin{enumerate}[(i)]
\item \textbf{(Local orbifold curve)} For a smooth proper orbifold
curve $\mathcal{C}$ whose orbifold locus consists of a point with
residual gerbe $B\mu_{\scriptscriptstyle N+1}$, let 
\[
\XX = \Tot_{\mathcal{C}}(L_{1} \oplus L_{2})
\]
where $L_{1}, L_{2} \in \Pic(\mathcal{C})$ satisfying $L_{1} \otimes
L_{2} \cong K_{\mathcal{C}}$.
\item \textbf{(Local orbifold surface)} For a smooth proper orbifold
surface $\mathcal{S}$ such that $H^{1}(\mathcal{S},
\mathcal{O}_{\mathcal{S}}) =0$ and whose orbifold locus consists of a
point with residual gerbe $B\mu_{\scriptscriptstyle N+1}$, let
\[
\XX = \Tot_{\mathcal{S}}(K_{\mathcal{S}}).
\]
\end{enumerate}
\end{definition}

In both cases, the orbifold locus $\BB \subset \XX$ is a line in the
fiber over the orbifold point in the curve or surface. In particular,
it is a trivial $B \mu_{\scriptscriptstyle N+1}$-gerbe over a curve
$B \isom \CC$.

\begin{remark}\label{remark: connectedness assumption 2}
The assumption that $B$ is connected is purely for notational simplicity. The results in this paper will immediately generalize to the case of finitely many disjoint curves, each isomorphic to $\CC$, along which $X$ has transverse singularities of $A$-type. 
\end{remark}

Just as in the ordinary case, our approach to defining GV invariants
passes through the PT theory of $\XX$. Let $\Lambda_{A_{N}}$ denote
the $A_{N}$ root lattice and $\LambdaWeightAN$ denote the dual
(weight) lattice, which contains $\Lambda_{A_{N}}$ as a
sublattice. The numerical $K$-theory of orbifolds is more delicate, and
the reader must consult Section \ref{section: GV invariants through PT
theory} for full details, but here we summarize the key points:  
\begin{itemize} 
\item The numerical group of point classes $N_{0}(\XX)$ is naturally
isomorphic to $\ZZ\oplus\Lambda_{A_N}$, see Lemma~\ref{lemma: FM with
N_exc and N_0}. An element $(n, 0)$ corresponds to $n[\OO_{\pt}]$
where $\pt$ is a generic point. An element $(0,v)$ corresponds to
$-[\OO_{b} \otimes R_{v}]$ where $b \in B$ and $R_{v}$ is a virtual
representation of $\mu_{\scriptscriptstyle N+1}$ associated to $v \in
\Lambda_{A_{N}}$.  
\vskip1ex
\item Let $N_{\leq 1}(\XX)$ be the numerical group of classes of
sheaves with proper support of dimension at most 
one. We define the subgroup 
\[
N_{1}(\XX) \subset N_{\leq 1}(\XX)
\otimes_{\ZZ} \QQ
\]
to be the classes with Euler characteristic zero, and orthogonal under
the Euler pairing to $K(\BB )$, the Grothendieck group of sheaves with
(not necessarily compact) support on the orbifold locus $\BB \subset
\XX$. Although our definition of $N_{1}(\XX )$ using orthogonality
necessitates passing to rational coefficients, we prove in
Lemma~\ref{lemma: basic K-theory lemma} that it suffices to enlarge
the root lattice to the weight lattice:
\[
N_{\leq 1}(\XX) \subset
N_{1}(\XX) \oplus \ZZ \oplus \LambdaWeightAN .
\]
\end{itemize}

As a consequence, there are (possibly empty) fine moduli spaces of
orbifold stable pairs on $\XX$ indexed by triples 
\[
(\beta, n, v) \in N_{1}(\XX) \oplus \ZZ \oplus \LambdaWeightAN .
\]
\begin{remark}[Integrality convention]\label{remark: integrality convention}
Throughout, we take moduli spaces in non-integral total classes to be
empty and the corresponding PT, GW, and GV invariants to be zero. Thus,
in a decomposition such as $\beta+v$, the individual summands need not
be integral.
\end{remark}
Just as in the ordinary case, we define via Behrend-weighted Euler characteristics
the PT invariants $\PT_{\beta, n, v}(\XX)$ which are assembled into
the PT partition function of $\XX$ 
\begin{equation}\label{eqn: PT partition function on local orbifold 2}
\ZZZ_{\XX}^{\PT}(Q, y, \bm{w}) = \sum_{\beta \in N_{1}(\XX)}\; \sum_{n
\in \ZZ}\; \sum_{v \in \LambdaWeightAN} \PT_{\beta, n, v}(\XX)
Q^{\beta} y^{n} \bm{w}^{v}.
\end{equation}
Here we introduce monomials $\bm{w}^{v} = w_{1}^{v_{1}}\cdots
w_{N}^{v_{N}}$ associated to each $v=\sum_{i}v_{i}\alpha_{i} \in
\LambdaWeightAN$ where $\alpha_{i}$ are the simple roots which form a
$\QQ$-basis for $\LambdaWeightAN$. These monomials freely generate the group
ring $\ZZ[\LambdaWeightAN]$ (see Definition \ref{defn: group ring
definition}) which carries an action by the Weyl group $W \cong
S_{N+1}$. By \cite[Theorem 23.24]{FultonHarris}, the ring of
Weyl-invariant elements is a polynomial ring 
\begin{equation}\label{eqn: Weyl invariant ring and the Phi's 2}
\ZZ[\LambdaWeightAN]^{W} \isom \ZZ[\Phi_{1}(\bm{w}), \ldots, \Phi_{N}(\bm{w})]
\end{equation}
with generators $\Phi_{1}(\bm{w}), \ldots, \Phi_{N}(\bm{w})$
given by the Weyl orbits of the fundamental weights. See Section
\ref{sec: root theory section} for examples of these generators for
small $N$.

The following is one of our main results in this paper. It generalizes
Theorem \ref{thm: Bridgeland and Toda Theorems} and establishes the
general structural symmetry present in $\ZZZ_{\XX}^{\PT}(Q, y,
\bm{w})$. 
\begin{introtheorem}\label{thm: main symmetries and structure of partition functions 2}
For all $\beta \in N_{1}(\XX)$ the formal series $\Coef_{Q^{\beta}}
\ZZZ_{\XX}^{\mathsf{PT}}(Q, y, \bm{w})$ is the Laurent expansion in
$y$ of an element 
\[
f_{\beta}(y, \bm{w}) \in \QQ(y)^{y \leftrightarrow y^{-1}} \otimes_{\ZZ} \ZZ[\LambdaWeightAN]^{W}.
\]
By (\ref{eqn: Weyl invariant ring and the Phi's 2}) we equivalently have $f_{\beta}(y, \bm{w}) \in \QQ(y)^{y \leftrightarrow y^{-1}} [\Phi_{1}(\bm{w}), \ldots, \Phi_{N}(\bm{w})]$. 
\end{introtheorem}

The symmetry established in Theorem \ref{thm: main symmetries and
structure of partition functions 2} is necessary for us to make the
following central definition of this paper: 
\begin{introdefinition}\label{defn: Main GVPT Defn 2}
Let $\XX$ be a local orbifold CY3 of $A_{N}$-type.  For a function
$f(\bm{w})$, we define $\Psi_{d}(f)(\bm{w}) = f(w_{1}^{d}, \ldots,
w_{N}^{d})$. The $\bm{m}$-graded GV invariants $\nrefined$ are
uniquely defined
by the formula
\begin{equation}\label{eqn: Main GV Conjecture via PT 2}
\begin{split}
& \log \ZZZ_{\XX}^{\PT}(Q, y, \bm{w}) = \sum_{\substack{ \beta \in N_{1}(\XX) \\ \beta \neq 0}} \,\, \sum_{g\in \ZZ } \,\, \sum_{d \geq 1} \,\,  \Psi_{d}\big(N_{\beta, g}(\bm{w})\big) \psi_{-(-y)^{d}}^{g-1} \, \frac{Q^{d \beta}}{d} \\
& N_{\beta, g}(\bm{w}) = \sum_{\bm{m} \geq 0} \nrefined \, \prod_{i=1}^{N} \Phi_{i}(\bm{w})^{m_{i}} \in \QQ[\LambdaWeightAN]^{W}
\end{split}
\end{equation}
where the latter sum is over $\bm{m} = (m_{1}, \ldots, m_{N}) \in
\ZZ_{\geq 0}^{N}$, and where $\psi_{x} = (x^{\frac{1}{2}} +
x^{-\frac{1}{2}})^{2}$.  
\end{introdefinition}

To establish key properties of the $\bm{m}$-graded GV invariants $\nrefined $ which
follow from this definition, it is illustrative to pass to the crepant
resolution $\pi : Y \to X$ of the coarse space (see Section
\ref{subsection: Derived McKay and Ktheory}). We know $Y$ is a smooth
CY3 containing an $A_{N}$-configuration of ruled surfaces over the
base $B \cong \CC$. Let $C_{1}, \ldots, C_{N}$ be the classes of the
exceptional curves in the fibers of the ruled surfaces. We identify
the simple roots $\alpha_{i}\in \Lambda_{A_{N}}$ with elements in
$N_{1}(Y)$ by
\[
\alpha_{i} = [\OO_{C_{i}}(-1)].
\]

By Definition \ref{defn: Main GVPT Defn 2}, the $\bm{m}$-graded GV
invariants $\nrefined $ are potentially non-zero for $g<0$. This is
addressed by the following

\begin{introtheorem}\label{thm: main structure thm of the GV invariants 2}
Let $\XX$ be a local orbifold CY3 of $A_{N}$-type with crepant
resolution $Y$. The $\bm{m}$-graded GV invariants are given in terms
of the ordinary GV invariants of $Y$ by the formula 
\begin{equation}\label{eqn: GV invariants of XX and Y 2}
\sum_{\bm{m} \geq 0} \nrefined \, \prod_{i=1}^{N}
\Phi_{i}(\bm{w})^{m_{i}} = \sum_{v \in \LambdaWeightAN}
n^{Y}_{g}\big(\FM^{-1}(\beta) + v\big) \, \bm{w}^{v} 
\end{equation}
where $\FM : N(Y) \to N(\XX)$ is the Fourier--Mukai transform
(cf. \S~\ref{subsection: Derived McKay and Ktheory})
and all but finitely many terms in the sum are zero. 
It follows that the $\bm{m}$-graded GV invariants are all integers,
and for fixed $\beta \in N_{1}(\XX)$ we have $\nrefined = 0$ for all
but finitely many pairs $(g, \bm{m})$ with $g \geq 0$.  
\end{introtheorem}

\begin{remark}
We do not know how to formulate a theory of Gopakumar-Vafa invariants
for an arbitrary quasi-projective CY3 $\XX$ of transverse
$A_{N}$-type. The issue is that if $\beta$ is a curve class on $\XX$
with a representative containing a component in the orbifold locus,
then Theorem~\ref{thm: main symmetries and structure of partition
functions 2} is false. However, our theory does extend to more general
orbifold CY3s $\XX$ if we restrict to curve classes transverse to the
orbifold locus. For example, if $\XX$ fibers over a curve or surface
and the orbifold locus is not in a fiber, then our theory works for
fiber curve classes. Namely, Theorems~\ref{thm: main symmetries and
structure of partition functions 2} and \ref{thm: main structure thm
of the GV invariants 2} hold for such $\beta $. 
\end{remark}

\subsection{The Total GV Invariants and Quantum Dimension} 
As a consequence of the finiteness property of Theorem \ref{thm: main
structure thm of the GV invariants 2}, we define the total GV
invariants to be the finite sum 
\begin{equation}\label{eqn: total GV invariants}
n_{g}^{\XX}(\beta) = \sum_{\bm{m}} \nrefined.
\end{equation}
In Section \ref{sec: The GV invariants of XX} we prove the following
Corollary of Theorem~\ref{thm: main structure thm of the GV invariants 2},
which shows that the total GV invariants arise from $N_{\beta,
g}(\bm{w})$ by specializing each $w_{k}$ to a primitive $(N+2)$-th
root of unity:  
\begin{corollary}\label{cor: unrefined PT/GV specialization 2}
Let  $\zeta =\exp\big(\frac{2 \pi i}{N + 2}\big)$. Then the total GV
invariants are given by 
\begin{equation*}
n^{\XX}_{g}(\beta) = \sum_{v \in \LambdaWeightAN} \zeta^{l(v)}\,
n^{Y}_{g}\left(\FM^{-1}(\beta) + v \right)
\end{equation*}
where $l(v)= \sum_{i}v_{i}$. 
\end{corollary}

\begin{remark}
In the context of representation theory, $N_{\beta, g}(\bm{w}) \in
\ZZ[\LambdaWeightAN ]^{W}$ can be viewed as the character of a
(virtual) representation of $\text{SL}_{N+1}(\CC)$. The specialization
$w_{k}= \zeta $ is of interest in this
context \cite{Gyenge-Nemethi-Szendroi,Nakajima2021} and by Corollary 
\ref{cor: unrefined PT/GV specialization 2} we note that the
total GV invariants coincide with the \emph{quantum dimension} of the
corresponding virtual representation.  
\end{remark}

\subsection{Local orbifold $K3$ surfaces} 

Given a smooth algebraic $K3$ surface
$S$ and a non-zero effective curve class $\beta \in \Pic(S)$, one can make
sense of the ordinary GV invariants of $S$ in the class $\beta$ by constructing a local $K3$ surface $X = S \times \CC$. These invariants are given by the celebrated Katz-Klemm-Vafa (KKV) formula
\cite{KKV,PandharipandeThomas2016}.

As an example of our theory, we generalize the KKV formula to local orbifold $K3$ surfaces of the following type:
\begin{definition}\label{defn: orbifold K3 Definition 2}
Let $\XX = \Sorb \times \CC$ be a local orbifold $K3$ surface with $\Sorb$ an orbifold $K3$ surface satisfying the following conditions:
\begin{enumerate}[(i)]
\item The coarse space $S$ contains an isolated singular point of $A_{N}$-type, and is smooth at all other points
\item If $\pi : \overline{S} \to S$ is the minimal resolution, we have
\[
\Pic(\overline{S}) = \pi^{*}\Pic(S) \oplus \Lambda_{A_{N}}(-1).
\]
\end{enumerate} 
\end{definition}
Equivalently, $S$ is a singular $K3$ surface with an isolated singular
point of $A_{N}$-type and such that all algebraic Weil divisors on $S$
are Cartier. There are many examples of such $\Sorb $ for $1 \leq N
\leq 19$, where $S$ is obtained by contracting an $A_{N}$ configuration of curves
in a smooth $K3$ surface.

\begin{remark}
The assumption that there is one singular point is purely for
notational simplicity. One can easily generalize to many isolated
singular points of $A$-type. The assumption on $\Pic(\overline{S})$
can also be dropped, but it is less trivial to do so. In general,
$\Pic(\overline{S})$ is an overlattice of $\pi^{*}\Pic(S) \oplus
\Lambda_{A_{N}}(-1)$ and the theta functions arising, for example in
(\ref{eqn: orb KKV statement}), would be replaced by certain shifted
theta functions. We will not pursue this here.
\end{remark}

In this case, the relevant numerical group of curve classes is $N_{1}(\XX) = p^{*}\Pic(S)$, where $p : \Sorb \to S$ is the canonical map to the coarse space. 

\begin{definition}
As a result of using Behrend-weighted Euler characteristics to develop the theory, there is a global sign which we want to define away:
\[
\Snrefined \coloneqq -\nrefined.
\]
In this example, we will call $\Snrefined$ the $\bm{m}$-graded GV invariants. 
\end{definition}

We regard the following as the analogue of the KKV formula \cite{KKV} for the local orbifold $K3$ surfaces of Definition \ref{defn: orbifold K3 Definition 2}:
\begin{introtheorem}\label{thm: orbifold KKV formula}
For $\beta \in N_{1}(\XX) = p^{*}\Pic(S)$, the $\bm{m}$-graded GV invariants $\Snrefined$ are given by the formula
\begin{equation}\label{eqn: orb KKV statement}
\sum_{g \geq 0} \sum_{\bm{m} \geq 0} \Snrefined \psi_{y}^{g-1} \prod_{i=1}^{N} \Phi_{i}(\bm{w})^{m_{i}} = -\Coef_{q^{\frac{1}{2}\beta^{2}}}\bigg( \frac{\Theta_{A_{N}}(q, \bm{w})}{\phi_{10,1}(q, -y)} \bigg).
\end{equation}
Here we define the Jacobi cusp form of weight $10$ and index $1$
\begin{equation}\label{eqn: Jacobi form weight 10 index 1}
\phi_{10, 1}(q, y) = q \, y^{-1}(1-y)^{2} \prod_{m=1}^{\infty} (1-q^{m})^{20}(1-yq^{m})^{2}(1-y^{-1}q^{m})^{2}
\end{equation}
as well as the Jacobi theta function of the $A_{N}$ root lattice
\begin{equation}\label{eqn: AN theta function}
\Theta_{A_{N}}(q, \bm{w}) = \sum_{v \in \Lambda_{A_{N}}} q^{\frac{1}{2}\langle v, v \rangle} \bm{w}^{v}
\end{equation}
where $\langle \cdot , \cdot \rangle$ is the (positive-definite) bilinear pairing on $\Lambda_{A_{N}}$. In particular, the integrality and finiteness of the $\bm{m}$-graded GV invariants holds. 
\end{introtheorem}

By setting $y=-1$, we specialize to the $\bm{m}$-graded GV invariants of genus $0$
\begin{equation}\label{eqn: genus zero specialization}
\sum_{\bm{m} \geq 0} n^{\Sorb}_{0}(\beta; \bm{m}) \prod_{i=1}^{N} \Phi_{i}(\bm{w})^{m_{i}} = \Coef_{q^{\frac{1}{2}\beta^{2}}}\bigg( \frac{\Theta_{A_{N}}(q, \bm{w})}{\Delta(q)} \bigg)
\end{equation}
where $\Delta(q) = q \prod_{m=1}^{\infty} (1-q^{m})^{24}$ is the discriminant cusp form of weight $12$.

\begin{example}
Let $\mathcal{S}$ be the elliptically fibered orbifold $K3$ surface
with a single $A_{18}$ orbifold point and trivial Mordell-Weil
group---see entry 112 of \cite[Table 2]{ShimadaK3Paper}. Let $\beta$ be the
class of a smooth fiber. Then an analysis of the theta function of the
$A_{18}$ lattice yields:
\[
\Coef_{q^{0}}\bigg( \frac{\Theta_{A_{18}}(q, \bm{w})}{\Delta(q)} \bigg) = \Phi_{1}(\bm{w}) \Phi_{18}(\bm{w}) + 5.
\]
Therefore, the non-vanishing complete genus-$0$ GV invariants are given by
\[
n^{\Sorb}_{0}(\beta ; \bm{m} ) = 
\begin{cases} 
      1, & \bm{m} = (1,0, \ldots, 0,1) \\
      5, & \bm{m} = (0, \ldots, 0). \\
    \end{cases}
\]
Geometrically, this counts the 6 rational fibers: there is one $I_{1}$
fiber with the $A_{18}$ singularity at the node, and $5$ ordinary
$I_{1}$ fibers in the smooth locus.

One takeaway from this example should be that even though
$\{\Phi_{1}(\bm{w}), \ldots, \Phi_{N}(\bm{w}) \}$ is certainly not a
unique basis of $\ZZ[\LambdaWeightAN]^{W}$, using this basis to define
the $\bm{m}$-graded GV invariants counts the isolated rational curves
as you would expect.  
\end{example}

Gyenge-Nemethi-Szendroi (GNS) \cite{Gyenge-Nemethi-Szendroi} studied the Euler characteristics of Hilbert schemes of points on \emph{singular} surfaces. In particular, the local calculation \cite[Theorem 3.1]{Gyenge-Nemethi-Szendroi} can be easily adapted to the case of a singular $K3$ surface $S$ as above, resulting in
\begin{equation}\label{eqn: GNS Formula}
\begin{split}
& \sum_{n \geq 0} e\big( \Hilb^{n}(S) \big) q^{n-1} = \frac{\Theta_{A_{N}}^{\mathsf{GNS}}(q)}{\Delta(q)} \\
& \Theta_{A_{N}}^{\mathsf{GNS}}(q) =  \Theta_{A_{N}}(q, \bm{w}) \big|_{w_{k} = \exp\big(\tfrac{2 \pi i}{N+2}\big)}.
\end{split}
\end{equation}
It is quite interesting to observe that the variable specialization
present in Corollary \ref{cor: unrefined PT/GV specialization 2} is
identical to that of (\ref{eqn: GNS Formula}) discovered by GNS
independently in a different context.  Combining (\ref{eqn: genus zero
specialization}) and (\ref{eqn: GNS Formula}) with Corollary
\ref{cor: unrefined PT/GV specialization 2}, we immediately observe
the following:

\begin{corollary}\label{cor: orbifold Yau-Zaslow}
The genus zero total GV invariant of class $\beta$ is given by
\[
n^{\Sorb}_{0}(\beta ) = e\big(\Hilb^{\frac{\beta^{2}}{2}+1}(S) \big). 
\]
\end{corollary}
Note that the formula is the same as the celebrated Yau-Zaslow formula
\cite{Yau-Zaslow}, but using the Hilbert scheme of points on the
\emph{singular surface} $S$ (not the orbifold). It would be very
interesting to find a direct geometric proof of this formula.

\subsection{The Local Teardrop}

In Section \ref{subsection: local teardrop} we construct a toric
orbifold $\XX$ called the \emph{local teardrop}, which is an example
of a local orbifold CY3 of $A_{N}$-type and is the orbifold version of
the resolved conifold. It is defined to be the total
space of a rank two bundle 
\[
\XX = \Tot_{\mathcal{C}} \big( \OO(-p_{0}) \oplus \OO(-p_{\infty}) \big)
\]
over the orbifold curve $\mathcal{C} = \PP^{1}(N+1, 1)$, and we get transverse
$A_{N}$ orbifold structure along a single torus-invariant
curve isomorphic to $\CC$. All compactly supported effective classes are of the form
$d[\PP^{1}]$, a positive multiple of the zero section.  

\begin{proposition}\label{prop: Multiple Cover Formula for Local Teardrop}
The $\bm{m}$-graded GV invariants of the local teardrop $\XX$ are given by
\[
n^{\XX}_{g}(d[\PP^{1}] ; \bm{m} ) = 
\begin{cases} 
      1 & g=0, \,\,\,\, d=1, \,\,\,\, \bm{m} = (1,0, \ldots, 0) \\
      0 & \text{otherwise} \\
    \end{cases}
\] 
\end{proposition}

We will prove this proposition in Section~\ref{subsection: local teardrop} using the localization computations of \cite{Johnson-Pandharipande-Tseng-localPab}. As we will explain, for this example we perform the computation using orbifold Gromov-Witten theory, but one can arrive at the same result using the orbifold vertex technology of Bryan-Cadman-Young \cite{BryanCadmanYoung} on the PT side.

\subsection{GV Invariants via Gromov-Witten Theory}

An alternative approach to studying curves in a given target space is
through Gromov-Witten (GW) theory, where one considers certain stable
maps from curves to that space. A lot is known about GW theory of an
ordinary variety and the GW theory of an orbifold has been studied for
over $20$ years now, with the foundations established in
\cite{Abramovich-Graber-Vistoli-2008,OrbiGW-Abramovich-Vistoli,Chen-Ruan}.

In Section \ref{sec: orbifold GW section} we will formulate orbifold
GW theory for an orbifold CY3 $\XX$ of $A_{N}$-type, as well as the
crepant resolution conjecture in GW theory \cite{BryanGraber,Chen-Ruan}. In
particular, on account of the non-compactness, we fix a
$\CC^{*}$-action on $\XX$ and apply the localization of
Graber-Pandharipande \cite{GraberPandharipande} in order to define the
GW invariants. Let $\F^{\GW}_{\XX}(Q, \lambda, \bm{x})$ be the
resulting GW potential, and let $\bm{x} = (x_{1}, \ldots, x_{N})$ be
formal variables tracking twisted sector insertions, see Definition
\ref{defn: full orbi GW potential}. We will prove the following (which
is Corollary \ref{cor: Main GW corollary} in the main body of the
text):

\begin{introtheorem}\label{thm: intro version of GW/GV formula}
Let $\XX$ be a local orbifold CY3 of $A_{N}$-type which is one of the following forms:
\begin{enumerate}[(i)]
\item a local orbifold $K3$ surface (with the fiberwise $\CC^{*}$-action),
\item a local orbifold surface $\Sorb$ with a $\CC^{*}$-action induced
from a $\CC^{*}$-action on $\Sorb$,
\item a local orbifold curve (with the anti-diagonal fiberwise $\CC^{*}$-action).
\end{enumerate}
Assume that the GW Crepant Resolution Conjecture holds for $\XX$. Then
the total GV invariants, as defined via PT theory, are related to the
orbifold GW invariants of $\XX$ by
the formula
\begin{equation}\label{eqn: Main intro GW/GV equation}
\F^{\GW}_{\XX}(Q,\lambda ,\bm{x})\Big\rvert_{x_{j} = \frac{\pi}{N+2}\csc
\left(\frac{j\pi}{N+1} \right)} = (-1)^{\epsilon} \sum_{\beta \neq 0}\, \sum_{g \geq 0}\, \sum_{d
\geq 1}\, n^{\XX}_{g}(\beta) \, \left( 2 \sin \tfrac{d \lambda}{2} \right)^{2g-2}\,
\frac{Q^{d \beta}}{d}
\end{equation}
where $\epsilon = 1$ for case (i)
and $\epsilon =0$ for cases (ii) and (iii).
\end{introtheorem}

The GW theory of $\XX$ encodes the $\bm{m}$-graded GV invariants more
generally (see Theorem \ref{thm: Orb GV Multiple Cover}), but we want
to emphasize that Equation~\eqref{eqn: Main intro GW/GV equation} is
structurally identical to the original Gopakumar-Vafa formula
\cite{Gopakumar-Vafa} after setting the twisted sector variables
$\bm{x}$ to the curious special values $x_{j}= \frac{\pi}{N+2}\csc
\left(\frac{j\pi}{N+1} \right)$.


For the reader's convenience, we include the following table
summarizing notation:

\par\vfill
\newpage
\input{notation-table}


\section{Gopakumar-Vafa Invariants Through Orbifold PT
Theory}\label{section: GV invariants through PT theory}

Let us begin by developing our definition of the Gopakumar-Vafa
invariants of orbifolds by way of Pandharipande-Thomas (PT)
theory. For the reader's convenience we note that many of the early
sections are meant to lay the groundwork ultimately leading to our
main results and constructions in Sections \ref{subsection: PT Theory
of Y and Weyl Invariance}--\ref{sec: The GV invariants of XX}.

\subsection{The $A_{N}$ Root System and Representation
Theory}\label{sec: root theory section} The theory of sheaves on
orbifolds has many close connections to root systems and
representation theory. We collect here all of the basic results we
make use of later.
\begin{definition}\label{defn: defn of AN root lattice}
The $A_{N} $ root lattice is the sublattice
of $\ZZ^{N+1}$ given by:
\[
\Lambda_{A_{N}} = \bigg\{ \sum_{i=1}^{N+1} c_{i} e_{i} \in \ZZ^{N+1}
\,\, \bigg| \,\, \sum_{i=1}^{N+1} c_{i} = 0 \bigg\}  
\]
where $\{e_{1}, \ldots, e_{N+1}\}$ is a basis of $\ZZ^{N+1}$ with
bilinear form $\langle e_{i}, e_{j} \rangle = \delta_{ij}$ extended to $\Lambda_{A_{N}}$ by linearity. 
\end{definition}

It is well-known that in this case the weight lattice
$\LambdaWeightAN$ is the dual of the root lattice, and we have 
\[
\Lambda_{A_{N}} \subseteq \LambdaWeightAN \subseteq \tfrac{1}{N+1} \Lambda_{A_{N}}. 
\]
The simple roots of the $A_{N}$ root system are given by:
\begin{equation}\label{eqn: AN simple roots}
\alpha_{k} = e_{k} - e_{k+1}, \quad \quad  k=1, \ldots, N.
\end{equation}
They form a $\ZZ$-basis of $\Lambda_{A_{N}}$ and a $\QQ$-basis of
$\LambdaWeightAN$. The bilinear form on the root lattice allows us to
compute the Cartan pairing of $A_{N}$ to be
\begin{equation}\label{eqn: Cartan pairing}
\langle \alpha_i,\alpha_j \rangle=
\begin{cases}
2 & \quad i=j,\\
-1 & \quad |i-j|=1,\\
0 &  \quad |i-j|\ge 2.
\end{cases}
\end{equation}

We introduce formal variables $\bm{w} = (w_{1}, \ldots, w_{N})$ and
associate to $v=\sum_{i}v_{i}\alpha_{i} $ the monomial 
\[
\bm{w}^{v} \coloneqq w_{1}^{v_{1}} \cdots w_{N}^{v_{N}} .
\]
Note that for $v\in \LambdaWeightAN$, $v_{i}\in
\tfrac{1}{N+1}\ZZ$. This leads us to the following definition:
\begin{definition}\label{defn: group ring definition}
The group ring of the $A_{N}$ weight lattice is the ring defined by
\[
\ZZ [ \LambdaWeightAN ] = \bigoplus_{v \in \LambdaWeightAN} \ZZ \bm{w}^{v}, \quad\quad\quad \bm{w}^{v_{1}} \bm{w}^{v_{2}} = \bm{w}^{v_{1} + v_{2}}.
\]
Consequently, an element of $\ZZ [ \LambdaWeightAN ]$ is a finite integer linear combination of monomials $\bm{w}^{v}$ for $v \in \LambdaWeightAN$ and the ring operations are sums and products of such expressions. 
\end{definition}

Let $W \cong S_{N+1}$ be the Weyl group of the $A_{N}$ root lattice. The Weyl group is generated by the reflections $s_{1}, \ldots, s_{N}$ about the simple roots $\alpha_{1}, \ldots, \alpha_{N}$, respectively, and induces an action on the weight lattice $\LambdaWeightAN$. We can equivalently formulate this as an action on the group ring $\ZZ[\LambdaWeightAN]$, with the reflections acting on the formal variables $\bm{w}$ as:
\begin{equation}\label{eqn: simple root reflections}
\begin{split}
& s_{1}(\bm{w}) = (w_{1}^{-1}, w_{1}w_{2}, w_{3}, \ldots, w_{N}) \\
& s_{j}(\bm{w}) = (w_{1}, \ldots, w_{j-1}w_{j}, w_{j}^{-1}, w_{j}w_{j+1}, \ldots, w_{N}) \quad j=2, \ldots, N-1 \\
& s_{N}(\bm{w}) = (w_{1}, \ldots, w_{N-1}w_{N}, w_{N}^{-1})
\end{split}
\end{equation}

Let $\ZZ[\LambdaWeightAN]^{W} \subseteq \ZZ[\LambdaWeightAN]$ denote
the subring of elements of the group ring invariant under the action
of the Weyl group. The simple roots uniquely determine a set of
fundamental weights $\{\omega_{1}, \ldots, \omega_{N}\}$ which form a
$\ZZ$-basis of $\LambdaWeightAN$. Associated to each $\omega_{i}$ we
can explicitly construct a Weyl-invariant element in the group ring: 
\begin{equation*}
\Phi_{i}(\bm{w}) = \sum_{v \in W(\omega_{i})} \bm{w}^{v} \in \ZZ[\LambdaWeightAN]^{W}
\end{equation*}
where $W(\omega_{i})$ denotes the Weyl orbit of the fundamental weight $\omega_{i}$. 

\begin{theorem}[{\cite[Theorem 23.24]{FultonHarris}}]
\label{thm: crucial Fulton-Harris thm on Weyl invariants}
The ring $\ZZ[\LambdaWeightAN]^{W}$ is a polynomial ring generated by the Weyl-invariant elements $\Phi_{i}(\bm{w})$:
\begin{equation}\label{eqn: Weyl invariant ring and the Phi's}
\ZZ[\LambdaWeightAN]^{W} \isom \ZZ[\Phi_{1}(\bm{w}), \ldots, \Phi_{N}(\bm{w})].
\end{equation}
\end{theorem}

\begin{remark}
Under the standard ADE classification, $\SL_{N+1}(\CC)$ is the complex Lie group associated to the $A_{N}$ root system. If we let $V_{1}, \ldots, V_{N}$ denote the
irreducible representations of $\SL_{N+1}(\CC)$ with highest weight
$\omega_{1}, \ldots, \omega_{N}$, respectively, then the Weyl-invariant basis elements can be equivalently written as
\[
\Phi_{i}(\bm{w}) = \Char(V_{i}) \in \ZZ[\LambdaWeightAN]^{W}
\]
where $\Char : \Rep(\SL_{N+1}(\CC)) \to \ZZ[\LambdaWeightAN]$ is the character
homomorphism.
\end{remark}

Define formal variables $\bm{z}
= (z_{1}, \ldots, z_{N+1})$ satisfying $z_{1} \cdots z_{N+1} = 1$ such that
\[
w_{k} = z_{k} z_{k+1}^{-1}, \quad \quad k=1, \ldots, N.
\]

\begin{lemma}[{\cite[Eqn. 23.26]{FultonHarris}}]
\label{lemma: AN change of vars Weyl-invariance}
With the above change of variables the Weyl-invariant basis elements for the $A_{N}$ root lattice satisfy
\[
\Phi_{i}(\bm{w}) = \sigma_{i}(\bm{z}) 
\]
where $\sigma_{i}(\bm{z})$ is the $i$-th elementary symmetric function of $N+1$ variables, defined through the expansion
\begin{equation}\label{eqn: Defn of Elem Symm Funcs}
\sum_{i=0}^{N+1} \sigma_{i}(\bm{z}) t^{i} = \prod_{k=1}^{N+1} (1+z_{k}t).
\end{equation}
\end{lemma}

We provide here examples of the $\Phi_{i}(\bm{w})$ for small values of $N$:

\begin{itemize}

\item For $A_{1}$:
\[
\Phi_{1}(w) = w^{\frac{1}{2}} + w^{-\frac{1}{2}}
\]
\item For $A_{2}$:
\begin{equation*}
\begin{split}
& \Phi_{1}(w_{1}, w_{2}) = w_{1}^{-\frac{1}{3}}w_{2}^{-\frac{2}{3}}(1 + w_{2} + w_{1}w_{2})   \\[4pt]
& \Phi_{2}(w_{1}, w_{2}) = w_{1}^{-\frac{2}{3}}w_{2}^{-\frac{1}{3}}(1 + w_{1} + w_{1}w_{2})
\end{split}
\end{equation*}

\item For $A_{3}$:
\begin{equation*}
\begin{split}
& \Phi_{1}(w_{1}, w_{2}, w_{3}) =  w_{1}^{-\frac{3}{4}}w_{2}^{-\frac{2}{4}}w_{3}^{-\frac{1}{4}}(1 + w_{1} + w_{1}w_{2} + w_{1}w_{2}w_{3}) \\[4pt]
& \Phi_{2}(w_{1}, w_{2}, w_{3}) =  w_{1}^{-\frac{1}{2}}w_{2}^{-1}w_{3}^{-\frac{1}{2}}(1 + w_{2} + w_{2}w_{3}+w_{1}w_{2} + w_{1}w_{2}w_{3} + w_{1} w_{2}^{2}w_{3})  \\[4pt]
& \Phi_{3}(w_{1}, w_{2}, w_{3}) =  w_{1}^{-\frac{1}{4}}w_{2}^{-\frac{2}{4}}w_{3}^{-\frac{3}{4}}(1 + w_{3} + w_{2}w_{3} + w_{1}w_{2}w_{3})
\end{split}
\end{equation*}
\end{itemize}

We now turn to a proof of Corollary \ref{cor: unrefined PT/GV
specialization 2} from the Introduction.
\begin{proof}[Proof of Corollary \ref{cor: unrefined PT/GV
specialization 2}] We assume that Theorem~\ref{thm: main structure thm
of the GV invariants 2} holds. Then by Equation~\eqref{eqn: GV
invariants of XX and Y 2}, we need to show that
\[
\Phi_{i}(\bm{w})|_{w_{k}=\zeta} =1
\]
where $\zeta =\exp\left(\tfrac{2\pi i}{N+2} \right)$.

Under the change of variables $w_{k} = z_{k} z_{k+1}^{-1}$ for $k=1,
\ldots, N$, the specialization  $w_{k}=\zeta $ is induced by
\begin{equation*}
z_{k} = - \zeta^{-k}.
\end{equation*}
By Lemma \ref{lemma: AN change of vars Weyl-invariance} the above
change of variables transforms $\Phi_{i}(\bm{w})$ into the elementary
symmetric function $\sigma_{i}(\bm{z})$. So it remains to show that
making the substitution $z_{k}=-\zeta^{-k}$ results in
$\sigma_{i} = 1$. By (\ref{eqn: Defn of Elem Symm Funcs}), this is
equivalent to showing that
\[
\prod_{k=1}^{N+1}(1-\zeta^{-k}t) = 1+\dotsb +t^{N+1}
\]
which holds since $\{\zeta^{-k} \}_{k=1,\dotsc ,N+1}$ is the set of
non-trivial $(N+2)$-th roots of unity.
\end{proof}

\subsection{Preliminaries on Orbifold CY3s}\label{subsection: preliminaries on orbifolds and PT theory}

We now turn to establishing some generalities on orbifolds, $K$-theory,
and crepant resolutions. We refer the reader to
\cite{Beentjes-Calabrese-Rennemo, BryanCadmanYoung} for more details
on the background material.  

Let $\XX$ be a local orbifold curve or surface of $A_{N}$-type (see
Definition~\ref{defn: new local orbifold CY3 defn}).  Let $K(\XX )$
denote the Grothendieck group of coherent sheaves on $\XX$ and let
$K_{c}(\XX)$ denote the Grothendieck group of compactly supported
coherent sheaves on $\XX$. We define the Euler pairing
\begin{align*}
\chi : K_{c}(\XX )&\otimes K(\XX )\to \ZZ \\
E&\otimes F  \mapsto \chi (E,F) 
\end{align*}
where
\[
\chi (E,F)=  \sum_{i
\in \ZZ} (-1)^{i} \dim \Ext_{\XX}^{i}(E,F).
\]
Two classes $E_{1}, E_{2} \in K_{c}(\XX)$ are defined to be numerically equivalent 
\[
E_{1} \sim_{\num} E_{2}
\]
if $\chi(E_{1}, F) =
\chi(E_{2}, F)$ for all
$F \in K(\XX)$. Throughout this paper, the \emph{numerical
$K$-theory group} of $\XX$ is defined to be the (compactly supported)
Grothendieck group modulo numerical equivalence 
\[
N(\XX) = K_{c}(\XX)/ \sim_{\num}
\]
which is a free Abelian group of finite rank. 
\begin{definition}
Let $N_{\leq 1}(\XX) \subset N(\XX)$ be the subgroup generated by sheaves with compact support of dimension at most $1$. In addition, let $N_{0}(\XX) \subset N_{\leq 1}(\XX)$ be the subgroup generated by $0$-dimensional sheaves.
\end{definition}
\noindent We always have $[\OO_{\pt}] \in N_{0}(\XX)$ where $\OO_{\pt}$ is the structure sheaf of a generic point. There are additional classes in $N_{0}(\XX)$ coming from structure sheaves of points $p \in B$ in the orbifold locus twisted by representations of $\mu_{\scriptscriptstyle N+1}$.

\subsection{Crepant Resolutions and the Derived McKay Correspondence}\label{subsection: Derived McKay and Ktheory}

For $\XX$ a local orbifold CY3 of $A_{N}$-type, let $X$ be the corresponding coarse space, and let
\[
Y = \Hilb^{[\OO_{\pt}]}(\XX)
\]
be the Hilbert scheme parameterizing $0$-dimensional substacks of
$\XX$ with class $[\OO_{\pt}] \in N_{0}(\XX)$. By a result of
Bridgeland, King, and Reid, $Y$ is a smooth quasi-projective
Calabi-Yau threefold, and the natural proper map 
\[
\pi : Y \to X 
\]
is a crepant resolution of the coarse space \cite{Bridgeland-King-Reid, ChenTseng2008}. 

The derived McKay correspondence is a
derived equivalence between $\XX$ and $Y$ by way of the Fourier--Mukai
functor of \cite{Bridgeland-King-Reid}
\[
\FM : D(Y) \to D(\XX)
\]
which induces an isomorphism on numerical $K$-theory
\begin{equation}\label{eqn: FM on num K-theory}
\FM : N(Y) \xrightarrow{\sim} N(\XX).
\end{equation}
This isomorphism is an isometry with respect to the Euler pairing on
$\XX$ and the similarly defined Euler pairing on $Y$:
\begin{align*}
\chi : K_{c}(Y)&\otimes K(Y )\to \ZZ \\
E &\otimes F \mapsto \chi (E,F)
\end{align*}
where 
\[
\chi (E,F) =  \sum_{i \in \ZZ} (-1)^{i} \dim \Ext_{Y}^{i}(E,F).
\]

The crepant resolution contains an exceptional locus $\pi^{-1}(B)$ consisting of an $A_{N}$-configuration of $N$ ruled surfaces $D_{1}, \ldots, D_{N}$. More specifically, for all $i = 1, \ldots, N$ we have ruled surfaces
\[
\pi |_{D_{i}} : D_{i} \to B
\]
over a curve $B \cong \CC$ with generic fiber $C_{i} \isom \PP^{1}$. 
\begin{lemma}\label{lemma: Euler pairing lemma}
For all $i, j = 1, \ldots , N$ we have
\[
\chi\big(\OO_{C_{i}}(-1), \OO_{D_{j}} \big) = 
\begin{cases}
-2 & \quad i=j,\\
1 & \quad |i-j|=1,\\
0 & \quad |i-j|\ge 2
\end{cases}
\]
which is (minus) the Cartan pairing presented in (\ref{eqn: Cartan pairing}).
\end{lemma}

\begin{proof}
From the divisor exact sequence for $D_{j} \subset Y$, given by
\[
0 \to \OO_{Y}(-D_{j}) \to \OO_{Y} \to \OO_{D_{j}} \to 0,
\]
we therefore have
\begin{equation*}
\chi \big(\OO_{C_{i}}(-1), \OO_{D_{j}}\big)  = \chi\big(
\OO_{C_{i}}(-1), \OO_{Y} \big) - \chi \big( \OO_{C_{i}}(-1), \OO_{Y}(-D_{j})\big) .
\end{equation*}
Since both $\OO_{Y}$ and $\OO_{Y}(-D_{j})$ are locally-free, we have
\begin{equation*}
\begin{split}
\chi\big(\OO_{C_{i}}(-1), \OO_{D_{j}}\big) & =  - \chi\big(\OO_{C_{i}}(-1)\big) + \chi\big( \OO_{Y}(D_{j})|_{C_{i}} \otimes \OO_{C_{i}}(-1) \big) \\
& = \deg \big( \OO_{C_{i}}(D_{j})\big)
\end{split}
\end{equation*}

It remains to compute $ \deg \big( \OO_{C_{i}}(D_{j})\big)$ and show
that it coincides with (minus) the Cartan pairing. The divisors
$D_{i}$ and $D_{j}$ are disjoint for $|i-j|\geq 2$, while $D_i\cap
D_j$ is a section of both rulings when $|i-j|=1$. Hence
\[
\deg \big( \OO_{C_{i}}(D_{j})\big)
=
\begin{cases}
1, & |i-j|=1\\
0, & |i-j|\geq 2.
\end{cases}
\]
For $i=j$, since $C_{i}$ is a fiber of the ruling $D_{i} \to B$, we have $N_{C_{i}/D_{i} }\cong \OO_{C_{i}}$. The normal bundle sequence
\[
0\longrightarrow N_{C_{i}/D_{i}}
\longrightarrow N_{C_{i}/Y}
\longrightarrow N_{D_{i}/Y}|_{C_{i}}
\longrightarrow 0
\]
and adjunction give
\[
\deg\bigl(N_{D_{i}/Y}|_{C_{i}}\bigr)
=
\deg\bigl(\det N_{C_{i}/Y}\bigr)
=
\deg\bigl(K_{C_{i}}\otimes K_{Y}^{-1}|_{C_{i}}\bigr)
=
-2.
\]
Since \(N_{D_{i}/Y}\cong\OO_{Y}(D_{i})|_{D_{i}}\), this yields $ \deg
\big( \OO_{C_{i}}(D_{i})\big) =-2$ which completes the proof.
\end{proof}

We define $N_{\leq 1}(Y) \subset N(Y)$ to be the subgroup generated by
sheaves with compact support of dimension at most $1$. The
Fourier--Mukai functor does not preserve the filtration on numerical
$K$-theory by dimension of support. Consequently, it is not true in
general that $\FM$ induces an isomorphism of $N_{\leq 1}(Y)$ with
$N_{\leq 1}(\XX)$. It is for this reason one often restricts to
multi-regular classes \cite{Beentjes-Calabrese-Rennemo,
BryanCadmanYoung}. However, as we will explain shortly, it is the case
that for the specific class of orbifolds under consideration in this
paper, we will have an isomorphism $N_{\leq 1}(Y) \cong N_{\leq
1}(\XX)$ induced by $\FM$.

\subsection{Numerical $K$-theory of Points and Exceptional Curves on $\XX$ and $Y$}

Throughout, let $\XX$ be a local orbifold CY3 of $A_{N}$-type and let
$Y$ be the crepant resolution. With the ultimate goal of constructing
generating series of enumerative invariants, we will devote the next
few sections to giving an explicit description of the numerical
$K$-theory groups of $\XX$ and $Y$ in a manner compatible with the
Fourier--Mukai map $\FM$.

Let $\pi_{*} : N(Y) \to N(X)$ be the pushforward of numerical
$K$-theory classes to the coarse space. We define the group
$N_{\exc}(Y) \subset N_{\leq 1}(Y) $ of \emph{exceptional classes} as
follows
\[
N_{\exc}(Y) = (\pi_{*})^{-1}\big(N_{0}(X)\big)
\]
which consists of curve classes supported in the ruling of the
exceptional divisors of $\pi$, as well as multiples of the generic
point class $[\OO_{y}]$.

Recall that we denote by $\Lambda_{A_{N}}$ the $A_{N}$ root lattice
(Definition \ref{defn: defn of AN root lattice}). Also let $R_{i}$ be
the representation of $\mu_{N+1}$ associated to the simple root
$\alpha_{i}$, and let $b \in \BB $ be a generic point in the orbifold
locus $\BB \subset \XX$.

\begin{lemma}\label{lemma: FM with N_exc and N_0}
The groups $N_{0}(\XX)$ and $N_{\exc}(Y)$ are given explicitly in terms of generators as follows
\begin{align*}
& N_{0}(\XX) = \langle [\OO_{\pt}] , [\OO_{b} \otimes R_{1}], \ldots, [\OO_{b} \otimes R_{N}] \rangle \\
& N_{\exc}(Y) = \langle [\OO_{y}], [\OO_{C_{1}}(-1)], \ldots, [\OO_{C_{N}}(-1)] \rangle.
\end{align*}
The Fourier--Mukai map~\eqref{eqn: FM on num K-theory} induces an isomorphism which is given on these generators as follows:
\begin{align*}
\FM \colon N_{\exc}(Y) & \xrightarrow{\;\sim\;} N_{0}(\XX) \\
 [\OO_{C_{i}}(-1)] & \mapsto -[\OO_{b} \otimes R_{i}] \\
 [\OO_{y}] & \mapsto [\OO_{\pt}].
\end{align*}
\end{lemma}
\begin{proof}
This is proved in essentially the same way as the classical result of Kapranov-Vasserot
\cite[Theorem~2.3]{Kapranov-Vasserot} for surfaces.
\end{proof}

Using the above lemma, we may canonically identify 
\[
N_{\exc}(Y) \cong \ZZ \oplus \Lambda_{A_{N}}
\qquad\text{and}\qquad
N_{0}(\XX) \cong \ZZ \oplus \Lambda_{A_{N}},
\]
given explicitly as follows:
\begin{enumerate}[(a)]
\item On $N_{\exc}(Y)$, the identification sends the class $[\OO_{y}]$
to $(1,0)$ and the class $[\OO_{C_{i}}(-1)]$ to $(0, \alpha_{i})$,
where $\alpha_{i}$ is the $i$-th simple root of $A_{N}$.
\item On $N_{0}(\XX)$, the identification sends the class
$[\OO_{\pt}]$ to $(1,0)$ and the class $[\OO_{b}\otimes R_{i}]$ to
$(0, -\alpha_{i})$, where $R_{i}$ denotes the irreducible
representation of $\mu_{N+1}$ corresponding to $\alpha_{i}$.
\end{enumerate}
With respect to these identifications, $\FM$ is the identity on $\ZZ \oplus \Lambda_{A_{N}}$.

\subsection{Numerical $K$-theory of Non-exceptional Classes on $Y$}
We next want to understand the numerical groups of non-exceptional
 classes on $Y$, which is slightly more subtle. To start, we note
that the holomorphic Euler characteristic $\chi(-)$ induces a
canonical splitting over $\ZZ$ 
\[
N_{\leq 1}(Y) = N_{1}(Y) \oplus N_{0}(Y)
\]
where $N_{1}(Y)$ consists of classes with vanishing Euler
characteristic supported on curves, and $N_{0}(Y) \isom \ZZ$ is
generated by the class $[\OO_{y}]$.  

One might na\"{\i}vely hope to obtain a natural splitting of $N_{\leq
1}(Y)$ into a direct sum of $N_{\exc}(Y)$ and a group of classes
orthogonal to exceptional classes under the Euler pairing, but such a
splitting exists only over $\QQ$. More precisely, we regard the
following as the natural numerical group of non-exceptional classes:

\begin{definition}\label{defn: non-exceptional curve class defn}
Let $N_{1}(Y)^{\natural} \subset N_{1}(Y) \otimes  \QQ$ be the
group consisting of classes $F \in N_{1}(Y)
\otimes  \QQ$ such that $\chi (F,\OO_{D_{i}})=0
$ for all $i = 1, \ldots, N$.
\end{definition}

The inevitability of $\QQ$-coefficients can be understood as
follows. Taking the orthogonal complement to the exceptional divisors
within $N_{1}(Y)$ is a canonical way to isolate non-exceptional curve
classes. It is not, however, a primitive sublattice of $N_{1}(Y)$. We
therefore need to introduce denominators. But in order to get an
integral class in $N_{1}(Y)$ we must accordingly add a fractional
linear combination of exceptional curve classes. 

It is clear from Definition \ref{defn: non-exceptional curve class defn} and Lemma \ref{lemma: FM with N_exc and N_0} that we have
\begin{equation}\label{eqn: initial num K-theory inclusion}
N_{\leq 1}(Y) \subset N_{1}(Y)^{\natural} \oplus \ZZ \oplus (\Lambda_{A_{N}} \otimes  \QQ).
\end{equation}
The point of the following lemma (which we apply in Section
\ref{subsection: PT Theory of Y and Weyl Invariance} when constructing
generating series of enumerative invariants) is that in order to
capture all of $N_{\leq 1}(Y)$ it suffices to enlarge the root lattice
$\Lambda_{A_{N}}$ to the weight lattice $\LambdaWeightAN \subset
\Lambda_{A_{N}} \otimes  \QQ$.  

\begin{lemma}\label{lemma: Ktheory on resolution}
We have the inclusion
\begin{equation*}
N_{\leq 1}(Y)  \subset N_{1}(Y)^{\natural} \oplus \ZZ \oplus \LambdaWeightAN.
\end{equation*}
\end{lemma}

\begin{proof}
By (\ref{eqn: initial num K-theory inclusion}) an arbitrary element $E
\in N_{\leq 1}(Y)$ can be uniquely written as 
\[
E = \overline{E} + n[\OO_{y}] + \sum_{j=1}^{N} v_{j} [\OO_{C_{j}}(-1)]
\]
where $\overline{E} \in N_{1}(Y)^{\natural}$, $n \in \ZZ$, and
\[
v=\sum_{j=1}^{N}v_{j} \alpha_{j}  \in
\Lambda_{A_{N}} \otimes  \QQ.
\]
However, we require an integer-valued Euler pairing of $E$ with all
classes, which imposes strong constraints on $v \in \Lambda_{A_{N}}
\otimes  \QQ$. We consider the Euler pairing of $E$ with the
classes $\OO_{D_{i}}$. It is clear that the $\overline{E}$ and $n
[\OO_{y}]$ terms are annihilated in the pairing, and we are left with:
\[
\chi(E, \OO_{D_{i}}) = \sum_{j=1}^{N} v_{j}\,  \chi(\OO_{C_{j}}(-1) ,
\OO_{D_{i}}) = - \langle v , \alpha_{i} \rangle   
\]
where $\left\langle \cdot ,\cdot \right\rangle$ is the Cartan
pairing. The equality follows from Lemma \ref{lemma: Euler
pairing lemma}. By the definition of the dual lattice, the above
pairing being integral implies $v \in \LambdaWeightAN$.  
\end{proof}

\subsection{Numerical $K$-theory of Curves on $\XX$} 

One typically defines the group of multiregular classes $N_{\mr}(\XX)
\subseteq N_{\leq 1}(\XX)$ to be the image of $N_{\leq 1}(Y)$ under
$\FM$ \cite{Beentjes-Calabrese-Rennemo, BryanCadmanYoung}. If $\XX$ is a local orbifold CY3 of $A_{N}$-type,
the orbifold locus contains no proper one-dimensional
component. Hence, every compactly supported one-dimensional sheaf on
$\XX$ has generically trivial stabilizer along each component of its
support. It follows that every compactly supported curve class is
multi-regular, and therefore 
\[
N_{\mr}(\XX)=N_{\leq 1}(\XX).
\]

In Definition \ref{defn: non-exceptional curve class defn} of the
previous section we defined $N_{1}(Y)^{\natural}$, the group of
non-exceptional classes on the resolution $Y$. We use the Fourier--Mukai
isomorphism to define the corresponding group on the orbifold:
\begin{definition}\label{defn: definition of N_1 on the orbifold} 
$N_{1}(\XX) \coloneqq \FM \big( N_{1}(Y)^{\natural} \big) \subset
N_{\leq 1}(\XX) \otimes  \QQ$.
\end{definition}
 We will identify $N_{1}(\XX )$ and
$N_{1}(Y)^{\natural}$ using $\FM$ throughout. 

The presence of $\QQ$-coefficients is of course entirely parallel to
the previous section. In particular, we have 
\begin{lemma}\label{lemma: basic K-theory lemma}
For $\XX$ a local orbifold CY3 of $A_{N}$-type we have
\[
N_{\leq 1}(\XX) \subset N_{1}(\XX) \oplus \ZZ \oplus \LambdaWeightAN.
\]
\end{lemma}

\begin{proof}
Applying the inverse Fourier--Mukai functor $\FM^{-1}$, the statement reduces to the result of Lemma \ref{lemma: Ktheory on resolution}. 
\end{proof}

\subsection{The PT Theory of $Y$ and Weyl Invariance}\label{subsection: PT Theory of Y and Weyl Invariance} 

With the necessary background established, we turn now to our primary
interest of curve counting. For $\XX$ a local orbifold CY3 of
$A_{N}$-type and $Y$ the associated crepant resolution, the first goal
is to review the PT theory of $Y$ and prove a key structural result on
Weyl invariance (Theorem \ref{thm: Main Weyl invariance
theorem}). This will be necessary to define orbifold Gopakumar-Vafa
invariants (Definition \ref{defn: Main GVPT Defn}).

Recall that the exceptional locus $\pi^{-1}(B)$ of $\pi: Y \to X$
consists of an $A_{N}$ configuration of ruled surfaces $D_{i} \to B$
for all $i =1, \ldots, N$ with general fiber $C_{i} \isom \PP^{1}$ and
that we identify (Lemma~\ref{lemma: FM with N_exc and N_0}) the
simple root $\alpha_{i}$ with a numerical $K$-theory class
\[
\alpha_{i} = [\OO_{C_{i}}(-1)] \in N_{1}(Y).
\]
We may then regard vectors in the weight lattice as elements in
numerical $K$-theory
\[
v= \sum_{i=1}^{N} v_{i} \alpha_{i} \in N_{1}(Y) \otimes  \QQ .
\]
The main application of Lemma \ref{lemma: Ktheory on resolution} is to
conveniently package the discrete invariants of $F$. Associated to
each triple 
\[
(\beta, n, v) \in N_{1}(Y)^{\natural} \oplus \ZZ \oplus \LambdaWeightAN
\]
is a fine (typically non-proper) moduli scheme $P_{n}(Y, \beta +
v)$ parameterizing stable pairs such that $F$ has $K$-theory
class $(\beta, n, v)$.

We define the PT invariants to be Behrend-weighted Euler characteristics
\[
\PT_{\beta + v, n}(Y) = e\big( P_{n}(Y, \beta + v), \nu \big) \in \ZZ
\]
where $\nu : P_{n}(Y, \beta + v) \to \ZZ$ is Behrend's
constructible function \cite{KaiBehrend}. We then assemble the PT
invariants into a generating series in the usual way. 
\begin{definition}
The \emph{PT partition function} of $Y$ is the following formal generating series:
\begin{equation}\label{eqn: PT series defn}
\ZZZ_{Y}^{\PT}(Q, y, \bm{w}) = \sum_{\beta \in N_{1}(Y)^{\natural}} \, \sum_{n \in \ZZ} \;\;  \sum_{v \in \LambdaWeightAN} \, \PT_{\beta + v, n}(Y) \, Q^{\beta} y^{n} \bm{w}^{v}
\end{equation}
where $\bm{w}^{v} = w_{1}^{v_{1}}\cdots w_{N}^{v_{N}}$ for $v=\sum_{i}v_{i}\alpha_{i}$. 
\end{definition}

\begin{lemma}\label{lem: formula for Zexc}
The PT partition function restricted to curve classes supported in the
exceptional fibers of $\pi$ 
\begin{equation*}
\ZZZ_{\exc}(y, \bm{w}) = \sum_{n \in \ZZ} \, \sum_{v \in
\Lambda_{A_{N}}} \, \PT_{v, n}(Y) \, y^{n} \bm{w}^{v} 
\end{equation*}
is a formal generating series given explicitly by 
\begin{equation*}
\ZZZ_{\exc}(y, \bm{w}) = \prod_{1 \leq a \leq b \leq N} M(\bm{w}_{[a,b]}, -y)
\end{equation*}
where for $a,b \in [1, N]$ with $a \leq b$ we define $\bm{w}_{[a,b]} = w_{a} w_{a+1} \cdots w_{b}$ and 
\[
M(x, y) = \prod_{n=1}^{\infty} (1-xy^{n})^{-n}
\]
is the (weighted) MacMahon function. Note that we may also write
$\ZZZ_{\exc}$ as 
\[
\ZZZ_{\exc}(y,\bm{w}) = \prod_{v\in R^{+}}M(\bm{w}^{v },-y)
\]
where $R^{+}$ is the set of positive roots.
\end{lemma}

\begin{proof}
Since the curves of class $v$ all have reduced
support in the exceptional divisor, $\PT_{v,n}(Y)$ is determined by a
formal neighborhood of the exceptional divisor and thus we may replace
$Y$ with another threefold having the same exceptional divisor and
formal neighborhood. In particular, consider the local threefold model
\[
Y_{N} =  S_{N} \times \CC,
\]
which is the unique crepant resolution of
\[
\CC^3/\mu_{\scriptscriptstyle N+1},
\]
where $\mu_{\scriptscriptstyle N+1} \subset SL_3(\mathbb C)$ acts with
weights $(1,-1,0)$. Then the formal neighbourhood of the exceptional
divisor coincides with the formal neighbourhood of $Y_{N}$, so we have
\[
\PT_{v, n}(Y) = \PT_{v, n}(Y_{N}).
\]
Ben Young \cite[Theorem 1.7]{BenYoungDiagrams} computes the
Donaldson--Thomas (DT) partition function of $Y_{N}$ to be
\[
\ZZZ^{\DT}_{Y_{N}}(y,\bm{w}) = M(1,-y)^{N+1} \prod_{1\leq a\leq b\leq N} M(\bm{w}_{[a,b]},-y).
\]
The factor $M(1,-y)^{N+1}$ is the degree-zero DT contribution, which
is removed when passing from DT to PT theory using the PT/DT correspondence
\cite{Bridgeland-PTDT,Toda-PTDT}. We therefore have:
\[
\ZZZ_{\exc}(y, \bm{w}) = \ZZZ^{\PT}_{Y_{N}}(y,\bm w) = \prod_{1\leq a\leq b\leq N} M(\bm{w}_{[a,b]},-y)
\]
which completes the proof.
\end{proof}

Bryan and Steinberg \cite{Bryan-Steinberg} study $\pi$-relative PT
invariants, where $\pi : Y \to X$ is the crepant resolution. They
prove that the formal generating series of the $\pi$-relative PT
invariants is related to the standard PT theory by the equality of
series:
\begin{equation}\label{eqn: defn of BS partition function}
\ZZZ_{Y}^{\BS }(Q, y, \bm{w}) = \frac{\ZZZ_{Y}^{\PT}(Q, y, \bm{w})}{\ZZZ_{\exc}(y, \bm{w})}.
\end{equation}
For fixed $\beta \in N_{1}(Y)^{\natural}$, we define
\begin{equation}
\ZZZ^{\BS }_{\beta}(y, \bm{w}) \coloneqq \Coef_{Q^{\beta}} \ZZZ_{Y}^{\BS }(Q, y, \bm{w}).
\end{equation}
Recall that, as in Section \ref{sec: root theory section}, we denote
by $\ZZ[\LambdaWeightAN]$ the group ring of the weight lattice, which
carries an action by the Weyl group $W \cong S_{N+1}$ of $A_{N}$.  

We will make use of Beentjes-Calabrese-Rennemo
\cite{Beentjes-Calabrese-Rennemo} and Buelles-Moreira
\cite{Buelles-Moreira}, both of which apply wall-crossing techniques
for projective varieties and orbifolds. As remarked in \cite[Section
1.3.2]{Beentjes-Calabrese-Rennemo}, extending the wall-crossing
results to the quasi-projective setting requires the existence of a
$(-1)$-shifted symplectic structure \cite{PTVV} on the relevant open moduli
substacks. Strictly speaking, the shifted symplectic structure
required here is slightly more general than that of
\cite[Corollary~6.2]{BravDyckerhoff2021}, where they work with
complexes with proper support. However combining this with
\cite[Theorem~4.0.8]{PreygelShiftedSymplectic}, which treats perfect
complexes with fixed determinant on an ordinary quasi-projective
CY3, the existence of the required shifted symplectic structure follows.

Let $\XX$ be an orbifold CY3 with crepant resolution $Y$. Fixing a
generic ample divisor $A$ on the coarse space, let $L = p^{*}
\OO_{X}(A)$, where $p : \XX \to X$ is the canonical forgetful
map. Given any class $F \in N_{1}(\XX)$ there is an intersection
pairing with $L$ valued in $N_{0}(\XX)\otimes \QQ $ defined as follows: 
\[
F \star L \coloneqq F \otimes L - F .
\]
Twisting a one-dimensional sheaf by a line bundle doesn't change its
support, so the difference $F \otimes L - F$ indeed is a
zero-dimensional numerical $K$-theory class.


For $\beta \in N_{1}(\XX )$, we define
\[
\ZZZ^{\PT}_{\beta}(y,\bm{w})=\Coef_{Q^{\beta}} \ZZZ^{\PT}_{\XX}(Q,y,\bm{w}).
\]

We state the main results of Beentjes-Calabrese-Rennemo below,
adapted to our notation\footnote{We warn the reader that in the
published version of \cite{Beentjes-Calabrese-Rennemo}, Theorem B is
stated incorrectly in the introduction. The reader should refer to
Theorem B as stated in section 7, page 507 of \cite{Beentjes-Calabrese-Rennemo}.  }:

\begin{theorem}[{\cite[Theorems B and
C]{Beentjes-Calabrese-Rennemo}}]\label{thm: restatement of BCR
theorem} Let $\beta \in N_{1}(\XX )\cong N_{1}(Y)^{\natural}$. Then the two series $\ZZZ^{\BS
}_{\beta}(y, \bm{w})$ and $\ZZZ^{\PT}_{\beta }(y,\bm {w})$ are
possibly different Laurent expansions of the same rational function
$f_{\beta}(y, \bm{w})$.  Moreover,
\[
f_{\beta}(y,\bm {w}) = \sum_{D}\frac{g_{D}}{h_{D}} 
\]
where $g_{D}$ and $h_{D}$ are Laurent polynomials and the sum is over all
decompositions $D$ of $\beta =
\sum_{i=1}^{r} \beta_{i}$ into effective classes. Moreover $h_{D}$ is explicitly given by
\[
h_{D} = \prod_{i=1}^{r} \bigg( 1- \prod_{j=1}^{i} y^{2 n_{j}}
w_{1}^{2a_{1,j}}\cdots w_{N}^{2a_{N,j}}\bigg)^{2i}
\]
where the numbers $n_{j}, a_{1,j},\dotsc ,a_{N,j}$ are defined by
\[
\beta_{j} \star L = n_{j}[\OO_{\pt}] - \sum_{i=1}^{N} a_{i,j} [\OO_{b}
\otimes R_{i}]. 
\]

\end{theorem}

The following is one of the core results in this paper. We show that
for a local orbifold CY3 of $A_{N}$-type, the rational
function $f_{\beta}$ is invariant under the action of the Weyl group
and under inversion of $y$, the denominators of
$f_{\beta}$ are very simple, and $\ZZZ ^{\BS}_{\beta}$ and
$\ZZZ^{\PT}_{\beta}$ are equal as series:

\begin{theorem}\label{thm: Main Weyl invariance theorem}
Let $\XX$ be a local orbifold CY3 of $A_{N}$-type and let $Y$ be the
crepant resolution.  For $\beta \in N_{1}(\XX ) \cong
N_{1}(Y)^{\natural }$, we have an equality of series
\[
\ZZZ^{\PT}_{\beta}(y,\bm {w})= \ZZZ^{\BS }_{\beta}(y, \bm{w}).
\]
Moreover, this series is the Laurent expansion of a rational function
\[
f_{\beta}(y, \bm{w}) \in \QQ(y)^{y \leftrightarrow y^{-1}} \otimes \ZZ[\LambdaWeightAN]^{W},
\]
which by  Theorem~\ref{thm: crucial Fulton-Harris thm on Weyl
invariants} implies
\[
f_{\beta}(y,\bm{w})\in \QQ (y)^{y\leftrightarrow y^{-1}}\otimes \ZZ
[\Phi_{1},\dotsc ,\Phi_{N}].
\]
\end{theorem}

\begin{proof}
We divide the proof into three steps:
\begin{itemize}
\item [(A)] Equality of series $\ZZZ_{\beta}^{\PT}=\ZZZ_{\beta}^{\BS}$
and $f_{\beta}\in \QQ (y)\otimes \ZZ [\LambdaWeightAN ]$,
\item [(B)] invariance of $f_{\beta}$ under the Weyl group, and
\item [(C)] invariance of $f_{\beta}$ under $y\leftrightarrow y^{-1}$.
\end{itemize}

\smallskip

\textbf{(A)} As above, let $A$ be a generic ample divisor on
the coarse space $X$, let $L = p^{*}\OO_{X}(A)$, and let $\beta \in
N_{1}(Y)^{\natural}\cong N_{1}(\XX )$. Since the orbifold
locus of $\XX$ is necessarily non-proper, no effective decomposition
of $\beta$ can contain this non-proper curve as a component. Thus for
any effective decomposition of $\beta$ as in Theorem \ref{thm:
restatement of BCR theorem}, by the genericity of $A$, we can then
arrange that the intersection product $\beta_{j} \star L$ is supported
only on generic points. Therefore, we know that
$\ZZZ^{\BS}_{\beta}(y, \bm{w})$ and $\ZZZ^{\PT}_{\beta}(y, \bm{w})$
are Laurent expansions of a single rational function $f_{\beta
}(y,\bm {w})$ whose numerator is a Laurent polynomial and whose
denominator only involves $y$ and has factors of the form 
\[
(1 \pm y^{k})^{\ell}
\]
for $k \in \ZZ$ and $\ell \in \ZZ_{\geq 0}$. In particular, we have
\[
f_{\beta}(y, \bm{w}) \in \QQ(y) \otimes \ZZ[\LambdaWeightAN].
\]

Moreover, this then implies the equality of series $\ZZZ
^{\PT}_{\beta}=\ZZZ^{\BS}_{\beta}$ since both series have the
exponent of $y$ bounded from below, and such Laurent expansions of elements
of $ \QQ(y) \otimes\ZZ[\LambdaWeightAN]$ are unique\footnote{In
general, Laurent expansions in several variables depend on the choice
of a linear function on the exponents of the variables (see
\cite[\S~2.5.1]{Beentjes-Calabrese-Rennemo}). However, in our case,  as a consequence
of having the exponents of $w_{i}$ bounded from above and below, there
are only two Laurent expansions, depending on whether the exponent of
$y$ is bounded from above or below.}.

\smallskip
\textbf{(B)} We wish to show
\begin{equation*} 
f_{\beta}(y,g(\bm{w}))=f_{\beta}(y,\bm{w}), \text{ for all $g\in W$.}
\end{equation*}
It suffices to prove invariance under the
generators $s_{j}$ whose transformations on $\bm{w}$ are given by
Equation~\eqref{eqn: simple root reflections}. Fix
$j\in\{1,\ldots,N\}$. For
$v=\sum_{i=1}^{N}v_i\alpha_i$, set
\begin{equation*}
v_0=v_{N+1}=0,
\qquad
d_j(v)=v_{j-1}-2v_j+v_{j+1}.
\end{equation*}

We use the work of Buelles and Moreira \cite{Buelles-Moreira} who
consider the following setting. Suppose that $Y$ is any CY3 containing
a ruled surface $D\subset Y$ such that the fibers of the ruling may
be contracted by a birational morphism $Y\to X$. Then Buelles and
Moreira derive symmetries of the PT partition function of $Y$,
formulated as follows. Let $C$ be the fiber class of the ruling and
for $\delta \in N_{1}(Y)$ consider the Buelles-Moreira series:
\[
\ZZZ^{\BM ,D}_{\delta}(y,w) = \frac{\sum_{n,b} \PT_{\delta +bC,n}(Y)\,
y^{n}w^{b}}{\sum_{n,b} \PT_{bC,n}(Y)\,
y^{n}w^{b}}. 
\]
Their result says that $\ZZZ^{\BM ,D}_{\delta}$ is a Laurent expansion
of a rational function $g^{D}_{\delta }$ satisfying in particular
\[
g^{D}_{\delta}(y,w^{-1})=w^{-\delta \cdot D} g^{D}_{\delta}(y,w).
\]

We wish to apply this result to our setting for the divisor
$D_j\subset Y$ and for a class of the form $\beta+v\in N_1(Y)$,
where $\beta\in N_1(Y)^{\natural}$ and $v\in\Lambda_j$, with
\begin{equation*}
\Lambda_j
=
\left\{
v=\sum_{i=1}^{N}v_i\alpha_i\in\LambdaWeightAN
\text{ such that }0\leq v_j<1
\right\}.
\end{equation*}
Every $u\in\LambdaWeightAN$ can then be written uniquely as
\begin{equation*}
u=v+b\alpha_j,
\qquad
v\in\Lambda_j,\quad b\in\ZZ.
\end{equation*}
Then

\[
\ZZZ^{\BM ,D_{j}}_{\beta +v}(y,w_{j}) = M(w_{j},-y)^{-1} \sum_{b,n}
\PT_{\beta +v+b\alpha_{j},n}(Y)\, y^{n}w_{j}^{b}
\]
where we have used Lemma~\ref{lem: formula for Zexc} to write the
denominator explicitly. Since
\begin{equation*}
(\beta+v)\cdot D_j=d_j(v),
\end{equation*}
the Buelles--Moreira result says that
$\ZZZ^{\BM,D_j}_{\beta+v}$ is the Laurent expansion of a rational
function $g^{D_j}_{\beta+v}$ satisfying
\begin{equation*}
g^{D_j}_{\beta+v}(y,w_j^{-1})
=
w_j^{-d_j(v)}g^{D_j}_{\beta+v}(y,w_j).
\end{equation*}

Using the unique decomposition above, we express the formal series
$\ZZZ_{\beta}^{\BS}$ in terms of the formal series
$\ZZZ^{\BM,D_j}_{\beta+v}$ as follows:
\begin{equation}\label{eqn: ZBS written in terms of the ZBMs}
\ZZZ^{\BS}_{\beta}(y,\bm{w})  = F_{j}(y,\bm{w})\, \sum_{v\in \Lambda_{j}}
\ZZZ^{\BM ,D_{j}}_{\beta +v}(y,w_{j}) \,  \bm {w}^{v}
\end{equation}
where
\[
F_{j} = \frac{M(w_{j},-y)}{\ZZZ_{\exc}(y,\bm {w})}.
\]

Equation~\eqref{eqn: ZBS written in terms of the ZBMs} is an equality
of formal Laurent series.  We would like to interpret it in a ring
where the variable $w_{j}$ is a rational function variable.  To this
end, set
\begin{equation*}
K_j=\QQ((y))\big(w_j^{\frac{1}{N+1}}\big)
\end{equation*}
and define
\begin{equation*}
\widehat{\mathcal A}_j
=
K_j
\llbracket
w_1^{\frac{1}{N+1}},\ldots,
\widehat{w_j^{\frac{1}{N+1}}},
\ldots,w_N^{\frac{1}{N+1}}
\rrbracket
\left[
w_i^{-\frac{1}{N+1}}:i\neq j
\right].
\end{equation*}
Thus $\widehat{\mathcal A}_j$ consists of formal Laurent series in
the variables $w_i^{1/(N+1)}$, $i\neq j$, whose exponents are
bounded from below in each of these variables, with coefficients
rational in $w_j^{1/(N+1)}$ and Laurent series in $y$. The usual
effectivity bound on the PT classes implies that all the series below
belong to this completion.

For each $v\in\Lambda_j$, the Buelles--Moreira series
$\ZZZ^{\BM,D_j}_{\beta+v}(y,w_j)$ is the Laurent expansion at $w_j=0$ of the rational function
$g^{D_j}_{\beta+v}(y,w_j)$. Let
\begin{equation*}
\iota_j:
\QQ((y))\big(w_j^{\frac{1}{N+1}}\big)
\longrightarrow
\QQ((y))\big(\big(w_j^{\frac{1}{N+1}}\big)\big)
\end{equation*}
denote Laurent expansion at $w_j=0$, and extend $\iota_j$
coefficientwise to $\widehat{\mathcal A}_j$. This extension remains
injective.

Note that $w_j^{v_j}\in K_j$ for every $v\in\Lambda_j$.
Define
\begin{equation*}
H_j(y,\bm w)
=
\sum_{v\in\Lambda_j}
g^{D_j}_{\beta+v}(y,w_j)\bm w^v
\in\widehat{\mathcal A}_j.
\end{equation*}
Equation~\eqref{eqn: ZBS written in terms of the ZBMs}, together with the fact that
$\ZZZ^{\BS}_\beta$ is the Laurent expansion of $f_\beta$, gives
\begin{equation*}
\iota_j\big(f_\beta(y,\bm w)\big)
=
\iota_j\big(F_j(y,\bm w)H_j(y,\bm w)\big).
\end{equation*}
Since $\iota_j$ is injective, it follows that
\begin{equation}\label{eqn: f beta equals Fj Hj}
f_\beta(y,\bm w)
=
F_j(y,\bm w)H_j(y,\bm w)
\end{equation}
in $\widehat{\mathcal A}_j$.

The reflection $s_j$ acts on $\widehat{\mathcal A}_j$. Indeed,
Equation~\eqref{eqn: simple root reflections} gives
\begin{equation*}
s_j(\bm w^v)
=
w_j^{d_j(v)}\bm w^v.
\end{equation*}
Thus $s_j$ changes only the coefficient in $K_j$, leaving all
exponents in the completed variables unchanged. On the other hand,
the Buelles--Moreira functional equation gives
\begin{equation*}
g^{D_j}_{\beta+v}(y,w_j^{-1})
=
w_j^{-d_j(v)}g^{D_j}_{\beta+v}(y,w_j).
\end{equation*}
Consequently,
\begin{equation*}
\begin{split}
s_j\left(
g^{D_j}_{\beta+v}(y,w_j)\bm w^v
\right)
&=
g^{D_j}_{\beta+v}(y,w_j^{-1})
w_j^{d_j(v)}\bm w^v\\
&=
g^{D_j}_{\beta+v}(y,w_j)\bm w^v.
\end{split}
\end{equation*}
It follows term by term that
\begin{equation*}
s_j(H_j)=H_j.
\end{equation*}

We also have
\begin{equation*}
F_j(y,\bm w)
=
\prod_{\gamma\in R^+\setminus\{\alpha_j\}}
M(\bm w^\gamma,-y)^{-1}.
\end{equation*}
Every positive root $\gamma\neq\alpha_j$ contains at least one
simple root $\alpha_i$ with $i\neq j$. Hence every nonconstant
term in the corresponding factor has positive degree in at least one
of the completed variables, so this product is well-defined in
$\widehat{\mathcal A}_j$. Moreover by \cite[\S10.2, Lemma~B,
p.~50]{HumphreysLieAlgebras}, $s_j$ permutes
$R^+\setminus\{\alpha_j\}$, and therefore
\begin{equation*}
s_j(F_j)=F_j.
\end{equation*}
Applying $s_j$ to Equation~\eqref{eqn: f beta equals Fj Hj}, we obtain
\begin{equation*}
f_\beta(y,s_j(\bm w))
=
s_j(F_j)s_j(H_j)
=
F_jH_j
=
f_\beta(y,\bm w)
\end{equation*}
in $\widehat{\mathcal A}_j$. Finally,
\begin{equation*}
f_\beta(y,\bm w)
\in
\QQ(y)\otimes\ZZ[\LambdaWeightAN],
\end{equation*}
and the natural map from this uncompleted group algebra into
$\widehat{\mathcal A}_j$ is injective. Thus
\begin{equation*}
f_\beta(y,s_j(\bm w))
=
f_\beta(y,\bm w)
\end{equation*}
in $\QQ(y)\otimes\ZZ[\LambdaWeightAN]$.

\smallskip \textbf{(C)} It remains for us to prove that
$f_{\beta}(y^{-1},\bm{w})=f_{\beta}(y,\bm{w})$. We may prove this
using Beentjes, Calabrese, and Rennemo's duality symmetry
\cite[Proposition 7.18]{Beentjes-Calabrese-Rennemo}. Their result is a
symmetry of $f_{\beta}(y,\bm{w})$ induced by the dualizing functor
$\DD : N_{\leq 1}(\XX ) \to N_{\leq 1}(\XX )$ defined by
\[
\DD (F) = R\Hom (F,\OO_{\XX})[2]. 
\]
We claim that the action of $\DD $ on $(\beta ,n,v)\in N_{\leq 1}(\XX )\subset
N_{1}(\XX )\oplus \ZZ \oplus \LambdaWeightAN$ is given by
\[
\DD (\beta ,n,v) = (\beta ,-n, w_{0}(v))
\]
where $w_{0}\in W$ is the longest element in the Weyl group. Indeed,
since our splitting is orthogonal, $\DD$ acts separately on each
factor. The action on the first two factors is the same as the usual
duality for (non-orbifold) DT theory. The action on $\LambdaWeightAN$
follows from a straightforward computation:
\[
\DD (\alpha_{i}) = \DD (-[\OO_{b}\otimes R_{i}]) =[\OO_{b}\otimes
R^{\vee }_{i}] =[\OO_{b}\otimes R_{N+1- i}] = -\alpha_{{N+1-i}}.
\]
Then \cite[Proposition~7.18]{Beentjes-Calabrese-Rennemo} reads
\[
f_{\beta}(y,\bm{w}) = f_{\beta}(y^{-1},w_{0}(\bm{w})). 
\]
Since we have already proved that $f_{\beta}(y,\bm{w})$ is invariant
under the action of the Weyl group, the equality
\[
f_{\beta}(y,\bm{w}) = f_{\beta}(y^{-1},\bm{w})
\]
follows.
\end{proof}

\subsection{The Gopakumar-Vafa Invariants for $\XX$}\label{sec: The GV invariants of XX}

Having established in Theorem~\ref{thm: Main Weyl invariance theorem}
that $\ZZZ^{\PT}_{\beta}$ (and $\ZZZ^{\BS}_{\beta}$) is the Laurent
expansion of a rational function 
\[
f_{\beta}\in \QQ (y)^{y\leftrightarrow y^{-1}}\otimes \ZZ [\Phi_{1},\dotsc ,\Phi_{N}]
\]
we will no longer notationally distinguish between
$\ZZZ^{\PT}_{\beta}$ and $f_{\beta}$. In particular, we will regard
the coefficient of $Q^{\beta}$ in the series $\log
\ZZZ^{\PT}_{\XX}(Q,y,\bm{w})$ as an element of $\QQ
(y)^{y\leftrightarrow y^{-1}}\otimes \ZZ [\Phi_{1},\dotsc
,\Phi_{N}]$. This allows us to make the following

\begin{definition}\label{defn: Main GVPT Defn}
Let $\XX$ be a local orbifold CY3 of $A_{N}$-type. For a function
$f(\bm{w})$, we define $\Psi_{d}(f)(\bm{w}) = f(w_{1}^{d}, \ldots,
w_{N}^{d})$. The $\bm{m}$-graded GV invariants $\nrefined$ are defined by
the formula
\begin{equation}\label{eqn: Main GV Conjecture via PT}
\begin{split}
& \log \ZZZ_{\XX}^{\PT}(Q, y, \bm{w}) = \sum_{\substack{ \beta \in N_{1}(\XX) \\ \beta \neq 0}} \,\, \sum_{g \in \ZZ } \,\, \sum_{d \geq 1} \,\,  \Psi_{d}\big(N_{\beta, g}(\bm{w})\big) \psi_{-(-y)^{d}}^{g-1} \, \frac{Q^{d \beta}}{d} \\
& N_{\beta, g}(\bm{w}) = \sum_{\bm{m} \geq 0} \nrefined \, \prod_{i=1}^{N} \Phi_{i}(\bm{w})^{m_{i}} \in \QQ[\LambdaWeightAN]^{W}
\end{split}
\end{equation}
where $\psi_{x} = (x^{\frac{1}{2}} + x^{-\frac{1}{2}})^{2}$.  
\end{definition}

By M\"obius inversion, one can see that the above formula uniquely
defines the numbers $\nrefined$. A priori, $\nrefined$ could be
non-zero for negative $g$. However we have the following, which also
establishes Theorem
\ref{thm: main structure thm of the GV invariants 2} from the
Introduction.
\begin{theorem}\label{thm: main structure thm of the GV invariants}
Let $\XX$ be a local orbifold CY3 of $A_{N}$-type with crepant
resolution $Y$. The $\bm{m}$-graded GV invariants are given in
terms of the ordinary GV invariants of $Y$ by the formula
\begin{equation}\label{eqn: GV invariants of XX and Y}
\sum_{\bm{m} \geq 0} \nrefined \,
\prod_{i=1}^{N} \Phi_{i}(\bm{w})^{m_{i}} = \sum_{v \in
\LambdaWeightAN} n^{Y}_{g}(\beta +v) \, \bm{w}^{v}. 
\end{equation}
It follows that $\nrefined$ are all integers, and for fixed $\beta \in
N_{1}(\XX)\cong N_{1}(Y)^{\natural}$ we have $\nrefined = 0$ for all but finitely many pairs
$(g, \bm{m})$ with $g \geq 0$ and $\bm{m} \in \ZZ_{\geq 0}^{N}$.
\end{theorem}

\begin{proof}
By Theorem~\ref{thm: Main Weyl invariance theorem} and
Equation~\eqref{eqn: defn of BS partition function}, we have
\begin{equation}\label{eqn: log PT/BS equality}
\log \ZZZ_{\XX}^{\PT}(Q, y, \bm{w}) = \log \ZZZ^{\BS}(Q,y,\bm {w})=
\log \ZZZ_{Y}^{\PT}(Q, y, \bm{w}) - \log \ZZZ_{\exc}(y, \bm{w})
\end{equation}
where the equality is as formal power series in $Q$ whose coefficients
are elements of the ring $\QQ(y)^{y \leftrightarrow y^{-1}} \otimes
\ZZ[\LambdaWeightAN]^{W}$. On both sides of the equality, the
coefficient of $Q^{0}$ vanishes. By Definition \ref{defn: Main GVPT
Defn} we have
\begin{equation}\label{eqn: GV expansion in proof}
\log \ZZZ_{\XX}^{\PT}(Q, y, \bm{w}) = \sum_{\substack{ \beta \in
N_{1}(\XX) \\ \beta \neq 0}} \; \sum_{\substack{g \in \ZZ \\ d
>0}}\;\;    \sum_{\bm{m} \in \ZZ^{N}_{\geq 0}} \nrefined \prod_{i=1}^{N}
\Phi_{i}(\bm{w}^{d})^{m_{i}} \psi_{-(-y)^{d}}^{g-1} \frac{Q^{d
\beta}}{d}
\end{equation}
where $\bm{w}^{d} = (w_{1}^{d}, \ldots, w_{N}^{d})$. We similarly have
the ordinary GV expansion of $Y$
\begin{equation}\label{eqn: GV expansion on Y in proof}
\begin{split}
\log \ZZZ_{Y}^{\PT}(Q, y, \bm{w}) & - \log \ZZZ_{\exc}(y, \bm{w}) \\
& = \sum_{ \substack{ \beta \in N_{1}(Y)^{\natural} \\ \beta \neq 0}}
\; \sum_{v \in \LambdaWeightAN}\; \sum_{\substack{g \geq 0\\ d >0}}
n_{g}^{Y}( \beta + v) \psi_{-(-y)^{d}}^{g-1} \bm{w}^{dv} \frac{Q^{d
\beta}}{d}.
\end{split}
\end{equation}
For fixed genus $g \in \ZZ$, define
\begin{equation*}
\begin{split}
& A_{\XX, g}(Q, \bm{w}) = \sum_{\substack{ \beta \in N_{1}(\XX) \\ \beta \neq 0}} \sum_{ \bm{m} \in \ZZ_{\geq 0}^{N}}\nrefined \prod_{i=1}^{N} \Phi_{i}(\bm{w})^{m_{i}} Q^{\beta}       \\
& A_{Y, g}(Q, \bm{w}) = \sum_{ \substack{ \beta \in
N_{1}(Y)^{\natural} \\ \beta \neq 0}} \sum_{v \in
\LambdaWeightAN}n_{g}^{Y}( \beta +v) \bm{w}^{v} Q^{\beta}
\end{split}
\end{equation*}
and observe that from the vanishing of the ordinary GV invariants for negative genus, we have $A_{Y, g}(Q, \bm{w}) = 0$ if $g < 0$. 

By (\ref{eqn: GV expansion in proof}) and (\ref{eqn: GV expansion on Y in proof}), we therefore see that the equality in (\ref{eqn: log PT/BS equality}) can be written as
\begin{equation}\label{eqn: equality in terms of Adams operator}
\sum_{g \in \ZZ} \sum_{d \geq 1} \frac{1}{d} \widetilde{\Psi}_{d}(A_{\XX, g}\psi_{y}^{g-1})(Q, y, \bm{w}) = \sum_{g \geq 0} \sum_{d \geq 1} \frac{1}{d} \widetilde{\Psi}_{d}(A_{Y, g}\psi_{y}^{g-1})(Q, y, \bm{w})
\end{equation}
where for $d \geq 1$ we define the Adams operator $\widetilde{\Psi}_{d}(f)(Q, y, \bm{w}) \coloneqq f(Q^{d}, -(-y)^{d}, \bm{w}^{d})$ acting on a formal series $f(Q, y, \bm{w})$. In particular, note that $\widetilde{\Psi}_{d}(\psi_{y}^{g-1}) = \psi_{-(-y)^{d}}^{g-1}$. Applying the M\"{o}bius inversion operator
\[
\sum_{k \geq 1} \frac{\mu(k)}{k} \widetilde{\Psi}_{k}
\]
to both sides of (\ref{eqn: equality in terms of Adams operator}) and using the standard property of the M\"{o}bius function:
\[
\sum_{k \mid n} \mu(k)
=
\begin{cases}
1, & n=1\\
0, & n>1
\end{cases}
\]
we conclude that 
\[
\sum_{g \in \ZZ} A_{\XX, g}(Q, \bm{w}) \psi_{y}^{g-1} = \sum_{g \geq 0} A_{Y, g}(Q, \bm{w}) \psi_{y}^{g-1}.
\]
This is an equality of power series in $Q$ with coefficients in the
ring $\QQ(y)^{y \leftrightarrow y^{-1}} \otimes
\ZZ[\LambdaWeightAN]^{W}$. Since $\QQ(y)^{y \leftrightarrow y^{-1}} =
\QQ(\psi_{y})$, a comparison of coefficients gives $A_{\XX,g}(Q,
\bm{w}) = A_{Y, g}(Q, \bm{w})$ for all $g$. Finally, recalling that
$N_{1}(\XX)\cong N_{1}(Y)^{\natural}$, we identify the coefficient
of $Q^{\beta}$ in $A_{\XX, g}$ with the coefficient of
$Q^{\beta }$ in $A_{Y, g}$, which completes the proof.
\end{proof}

\subsection{Local Orbifold $K3$ Surfaces}\label{subsec: constr of local orbifold K3s}

As one of our main collections of examples in this paper, we compute
the $\bm{m}$-graded GV invariants for the local orbifold $K3$ surfaces
$\XX = \Sorb \times \CC$ of Definition \ref{defn: orbifold K3
Definition 2}. Recall that we assume $\Sorb$ is an orbifold $K3$
surface satisfying the conditions:
\begin{enumerate}[(i)]
\item The coarse space $S$ contains an isolated singular point of $A_{N}$-type
\item If $\pi : \overline{S} \to S$ is the minimal resolution, we have
\[
\Pic(\overline{S}) = \pi^{*}\Pic(S) \oplus \Lambda_{A_{N}}(-1).
\]
\end{enumerate} 
Let $Y = \overline{S} \times \CC$ be the crepant resolution of the
coarse space of $\XX$. Recalling that we always work with compactly
supported numerical $K$-theory, we have 
\begin{itemize}
\item $N_{1}(Y)^{\natural} = \pi^{*}\Pic(S)$
\item $N_{1}(Y) = \Pic(\overline{S}) = \pi^{*}\Pic(S) \oplus \Lambda_{A_{N}}(-1)$
\item $N_{1}(\XX) = p^{*}\Pic(S)$
\end{itemize}
where $p : \Sorb \to S$ is the canonical map to the coarse space. Our
assumption on $\Pic(\overline{S})$ ensures that we have an integral
splitting $N_{1}(Y) = N_{1}(Y)^{\natural} \oplus \Lambda_{A_{N}}(-1)$
with no need to introduce denominators. It follows from Definition
\ref{defn: orbifold K3 Definition 2} that $S$ is a singular $K3$
surface with an isolated singular point of $A_{N}$-type and all
algebraic Weil divisors on $S$ are Cartier.

As a result of using the Behrend-weighted Euler characteristic to
define the PT invariants on both $\XX$ and $Y$, there is a global
minus sign that we want to define away:
\[
\Snrefined \coloneqq -\nrefined, \quad \quad \quad n_{g}^{\overline{S}}(\beta) \coloneqq -n_{g}^{Y}(\beta).
\]
The invariants $n_{g}^{Y}(\beta)$ are determined by \cite[Theorem 6.3]{MaulikThomas2018}, where one indeed observes a minus sign added to the usual KKV formula. With our definition, we have
\begin{equation}\label{eqn: main ordinary KKV formula}
\sum_{g \geq 0 } n_{g}^{\overline{S}}(\beta) \psi_{y}^{g-1} = -\Coef_{q^{\frac{1}{2}\beta^{2}}} \bigg( \frac{1}{\phi_{10, 1}(q, -y)}\bigg),
\end{equation}
and in particular, the genus-zero invariants are all non-negative with this convention. 

We now give the proof of Theorem \ref{thm: orbifold KKV formula} as stated in the Introduction. In particular, we compute the $\bm{m}$-graded GV invariants $\Snrefined$ for the local orbifold $K3$ surfaces under consideration. 

\begin{proof}[Proof of Theorem \ref{thm: orbifold KKV formula}]

Starting with the equation (\ref{eqn: GV invariants of XX and Y 2}) we
multiply both sides by $\psi_{y}^{g-1}$, and sum over $g \geq 0$ to
get
\[
\sum_{g \geq 0} \sum_{\bm{m} \geq 0} \Snrefined \psi_{y}^{g-1}
\prod_{i=1}^{N} \Phi_{i}(\bm{w})^{m_{i}} = \sum_{v \in
\Lambda_{A_{N}}} \sum_{g \geq 0} n_{g}^{\overline{S}}(\beta +v)
\psi_{y}^{g-1} \bm{w}^{v}.
\]
Note that the root lattice appears instead of the
weight lattice precisely due to our assumption on
$\Pic(\overline{S})$.

We can simplify the right-hand side above using the ordinary
Katz-Klemm-Vafa formula for smooth $K3$ surfaces \cite{KKV,PandharipandeThomas2016}. In
particular, we use the remarkable property that
\[
n_{g}^{\overline{S}}(\beta +v)
\]
only depends on the curve class through the self-intersection. We have 
\[
(\beta +v)^{2} = \beta^{2} - \langle v, v \rangle
\]
where $\langle \cdot, \cdot \rangle$ is the positive definite bilinear
pairing for the $A_{N}$ root lattice. Plugging the KKV formula
(\ref{eqn: main ordinary KKV formula}) into the formula involving the
$\bm{m}$-graded GV invariants yields
\begin{equation*}
\begin{split}
\sum_{g \geq 0} \sum_{\bm{m} \geq 0} \Snrefined  \psi_{y}^{g-1} \prod_{i=1}^{N} \Phi_{i}(\bm{w})^{m_{i}} & = -\sum_{v \in \Lambda_{A_{N}}} \Coef_{q^{\frac{1}{2}(\beta^{2} - \langle v, v \rangle)}} \bigg( \frac{1}{\phi_{10, 1}(q, -y)}\bigg) \bm{w}^{v} \\
& = -\Coef_{q^{\frac{1}{2} \beta^{2}}}\bigg( \frac{\Theta_{A_{N}}(q, \bm{w})}{\phi_{10, 1}(q, -y)} \bigg) 
\end{split}
\end{equation*}
where the Jacobi theta function $\Theta_{A_{N}}(q, \bm{w})$ is defined in (\ref{eqn: AN theta function}). This completes the proof. 
\end{proof}

\section{Gopakumar-Vafa Invariants Through Orbifold GW Theory}\label{sec: orbifold GW section}

\subsection{Preliminaries on Orbifold Cohomology and GW Theory}\label{subsection: orbifold GW theory preliminaries}

Gromov-Witten (GW) theory is the mathematical formulation of
topological string theory, and is concerned with enumerating stable
maps from curves into a target space. In the case of an orbifold
target space $\XX$, the foundations of the theory were developed in
\cite{Abramovich-Graber-Vistoli-2008,OrbiGW-Abramovich-Vistoli,Chen-Ruan}. The
notion of a stable map is replaced by that of a \emph{twisted stable
map}, which is a representable morphism
\[
f : \CCC \to \XX
\]
whose domain is a twisted curve $\CCC$. There is a Deligne-Mumford stack $\overline{M}_{g,n}(\XX, \beta)$ parameterizing twisted stable maps of genus $g$ with $n$ marked points, representing the compactly supported class $\beta \in H_{2}(X, \QQ)$ on the coarse space. 

Associated to the $i$-th marked point, we have an evaluation map taking values not in $\XX$, but rather in the \emph{rigidified inertia stack}
\[
\text{ev}_{i} : \overline{M}_{g,n}(\XX, \beta) \to \overline{\mathcal{I}} \XX. 
\]
As usual, we are interested in pulling back cohomology classes by the evaluation maps. The cohomology of the rigidified inertia stack was studied by Chen-Ruan \cite{Chen-Ruan}, and has come to be known as \emph{orbifold cohomology}. As a vector space, we have
\[
H^{*}_{\orb}(\XX, \QQ) = H^{*}(\overline{\mathcal{I}} \XX, \QQ)
\]
though the grading includes a shift by ages. 

From here onward, let $\XX$ be a local orbifold CY3 of $A_{N}$-type (see Definition \ref{defn: new local orbifold CY3 defn}). The degree $2$ orbifold cohomology is given by
\[
H_{\orb}^{2}(\XX, \QQ) \isom H^{2}(X, \QQ) \bigoplus \QQ^{N}.
\]
The factor $\QQ^{N}$ is referred to as the \emph{twisted sector}. In
the above result we have used that the $A_{N}$ orbifold structure
implies that the components of the twisted sector are indexed by $\{1,
\ldots, N\}$ and are all of age $1$.

In general, the GW invariants of the orbifold $\XX$ are multilinear
functions $\langle \cdots \rangle_{g, \beta}^{\XX}$ on elements of
$H^{*}_{\orb}(\XX, \QQ)$ pulled back via the evaluation maps and
integrated against the virtual fundamental class. However, since the
orbifolds we consider are non-compact, we define the GW invariants
using Graber-Pandharipande localization \cite{GraberPandharipande}, as
we now recall.

Let $\XX$ be an orbifold with a $\CC^{*}$-action of the type described
in Theorem~\ref{thm: intro version of GW/GV formula}. The
$\CC^{*}$-action lifts to the moduli space of twisted stable maps, and in all
cases, the $\CC^{*}$-fixed locus $\overline{M}_{g,n}(\XX, \beta)^{\CC^{*}}$ is
proper. Graber-Pandharipande construct a virtual class
$[\overline{M}_{g,n}(\XX, \beta)^{\CC^{*}}]^{\vir}$ and define invariants as
follows:

\begin{definition}
Let $\XX$ be an orbifold CY3 with a $\CC^{*}$-action given by one of the
types in Theorem~\ref{thm: intro version of GW/GV formula}. Let $\gamma_{1}, \ldots,
\gamma_{N}$ be canonical generators of the twisted sector
$\QQ^{N}$. The \emph{GW invariants}, with insertions from the twisted
sector, are defined by
\[
\langle \gamma_{i_{1}},\dotsc , \gamma_{i_{n}}
\rangle_{g,\beta}^{\XX} = \int_{[\overline{M}_{g,n}(\XX,
\beta)^{\CC^{*}}]^{\vir}} \frac{j^{*}\text{ev}_{1}^{*} (\gamma_{i_{1}}) \cup
\ldots \cup j^{*}\text{ev}_{n}^{*}
(\gamma_{i_{n}})}{e(N^{\vir})},
\]
where $j:\overline{M}_{g,n}(\XX,
\beta)^{\CC^{*}}\into \overline{M}_{g,n}(\XX,
\beta) $ is inclusion and $e(N^{\vir})$ is the equivariant Euler class of the virtual
normal bundle $N^{\vir}$. Then since $\langle \gamma_{i_{1}},\dotsc , \gamma_{i_{n}}
\rangle_{g,\beta}^{\XX}$ is symmetric in its entries and $i_{k}\in
\{1,\dotsc ,N \}$, we collect terms and write the insertion as a
monomial
\[
\langle \gamma_{1}^{m_{1}}\dotsb  \gamma_{N}^{m_{N}}
\rangle_{g,\beta}^{\XX}
\]
where $n=\sum_{j=1}^{N}m_{j}$ and $m_{j}\geq 0$.
\end{definition}

\begin{definition}\label{defn: full orbi GW potential}
Let $\bm{x} = (x_{1}, \ldots, x_{N})$ be a set of formal variables tracking twisted sector insertions. The full orbifold GW potential is given by:
\[
 \F^{\GW}_{\XX}(Q, \lambda, \bm{x}) =  \sum_{\substack{\beta \in H_{2}(X, \QQ) \\ \beta \neq 0}}\, \sum_{g \geq 0}
 \F_{\XX,\beta }^{g}( \bm{x})\, Q^{\beta } \lambda^{2g-2}  
\]
where
\[
 \F_{\XX, \beta}^{g}(\bm{x})  = \sum_{\bm{m} } \langle
 \gamma_{1}^{m_{1}} \cdots \gamma_{N}^{m_{N}} \rangle_{g, \beta}^{\XX}
 \prod_{j=1}^{N} \frac{x_{j}^{m_{j}}}{m_{j}!}.
\]
\end{definition}

\subsection{The GW Crepant Resolution Conjecture}

Let $\pi : Y \to X$ be the crepant resolution of the coarse space of
$\XX$, see Section \ref{subsection: Derived McKay and Ktheory}. Recall
the exceptional locus of $\pi$ consists of an $A_{N}$-configuration of
ruled surfaces $D_{1}, \ldots, D_{N}$ over the non-compact base $B$,
and we denote by $C_{1}, \ldots, C_{N}$ the respective generic fibers,
with $C_{i} \cong \PP^{1}$. In Definition \ref{defn: non-exceptional
curve class defn} we introduced the group $N_{1}(Y)^{\natural} \subset
N_{1}(Y) \otimes \QQ$ characterized by classes with trivial Euler
characteristic and vanishing Euler pairing with $\OO_{D_{i}}$ for $i
=1, \ldots, N$.
\begin{definition}\label{defn: H2(Y)perp}
Denote by $H_{2}(Y)^{\natural}$ the image of the cycle class map
\[
N_{1}(Y)^{\natural} \hookrightarrow H_{2}(Y, \QQ)
\]
associating the class of a one-dimensional sheaf to the class of its support. 
\end{definition}
\noindent It is straightforward to verify that $H_{2}(Y)^{\natural}$
is the orthogonal complement to the exceptional locus $\{D_{1},
\ldots, D_{N}\}$.

We then identify $v=\sum_{i}v_{i}\alpha_{i}$ with its corresponding
curve class:
\[
\sum_{i=1}^{N} v_{i} [C_{i}] \in H_{2}(Y, \QQ).
\]

The $T = \CC^{*}$ action on $\XX$ lifts canonically to $Y$, and the
induced $T$-fixed locus of the moduli space of stable maps is
proper. Completely analogous to the previous section, one can
therefore use Graber-Pandharipande localization to define the GW
invariants $\langle \,\,\, \rangle_{g, \beta +v}^{Y} \in \QQ$ where
$\beta \in H_{2}(Y)^{\natural}$ and $v\in \LambdaWeightAN$.  We note
that for an ordinary CY3, the GW invariants do not require
insertions.

\begin{definition}\label{defn: GW potentials F and Fn-exc}
The GW potential of $Y$ (for non-zero curve classes) is given by
\begin{equation*}
\F_{Y}^{\GW}(Q, \lambda, \bm{w}) = \sum_{\substack{\beta+v \in
H_{2}(Y) \\ \beta+v \neq 0}} \; \sum_{g \geq 0}\; \langle \,\,\, \rangle_{g, \beta + v }^{Y}
\; \lambda^{2g-2} Q^{\beta} \bm{w}^{v}.
\end{equation*}
The GW potential of $Y$ (for non-purely exceptional curve classes) is given by
\begin{equation*}
\F_{Y, \nexc}^{\GW}(Q, \lambda, \bm{w}) = \sum_{\substack{\beta \in
H_{2}(Y)^{\natural} \\ \beta \neq 0}} \; \sum_{g \geq 0}\;  \sum_{v \in
\LambdaWeightAN} \langle \,\,\, \rangle_{g, \beta + v }^{Y}
\; \lambda^{2g-2} Q^{\beta} \bm{w}^{v}.
\end{equation*}
\end{definition}

Let us now state the crepant resolution conjecture in Gromov-Witten
theory. This remains conjectural for general local orbifold CY3s of
$A_{N}$-type, but it has been proven in the toric setting
\cite{CCIT,Ross-CMP2015}.  

\begin{conjecture}[Bryan-Graber \cite{BryanGraber}]\label{conj: Bryan-Graber conj}
Let $\XX$ be a local orbifold CY3 of $A_{N}$-type with crepant
resolution $Y$. Performing the change of variables between $\bm{w} =
(w_{1}, \ldots, w_{N})$ and $\bm{x} = (x_{1}, \ldots, x_{N})$ defined
by $w_{i} = z_{i} z_{i+1}^{-1}$, where for all $k=1, \ldots, N+1$
\begin{equation}\label{eqn: CR change of variables}
z_{k} = -\omega^{-k + \frac{1}{2}} \exp \bigg( \frac{-1}{N+1} \sum_{j=1}^{N} \omega^{j(k-\frac{1}{2})} x_{j} \bigg) \,\,\,\,\,\,\,\,\, \omega = \exp\big(\tfrac{ 2 \pi i}{N+1} \big),
\end{equation}
we then have the equality of generating series
\begin{equation}\label{eqn: GW CRC eqn}
\F_{\XX}^{\GW}(Q, \lambda, \bm{x}) = \F_{Y, \nexc}^{\GW}(Q, \lambda, \bm{w}).
\end{equation}
\end{conjecture}

\noindent We note that the change of variables in (\ref{eqn: CR change
of variables}) is motivated by work of Bryan-Gholampour
\cite{Bryan-Gholampour} in the case of orbifolds of $A_{N}$-type.

\subsection{The Gopakumar-Vafa Invariants}\label{subsection: GVGW main body section}

Recall that in (\ref{eqn: Defn of Elem Symm Funcs}) we defined the
elementary symmetric functions $\sigma_{i}(\bm{z}) = \sigma_{i}(z_{1},
\ldots, z_{N+1})$ for $i=1, \ldots, N$ with the understanding that
$z_{1}\cdots z_{N+1} = 1$.

\begin{theorem}\label{thm: Orb GV Multiple Cover}
Assume the GW crepant resolution conjecture (GW CRC) holds for $\XX$,
a local orbifold CY3 of $A_{N}$-type which is one of the following
forms:
\begin{enumerate}[(i)]
\item a local orbifold $K3$ surface (with the fiberwise $\CC^{*}$-action),
\item a local orbifold surface $\Sorb$ with a $\CC^{*}$-action induced
from a $\CC^{*}$-action on $\Sorb$,
\item a local orbifold curve (with the anti-diagonal fiberwise $\CC^{*}$-action).
\end{enumerate}
Given a function $f(z_{1}, \ldots, z_{N+1})$, let $\Psi_{d}(f) =
f(z_{1}^{d}, \ldots, z_{N+1}^{d})$.  Then, with the change of
variables (\ref{eqn: CR change of variables}), the $\bm{m}$-graded GV
invariants are related to the orbifold GW invariants by the formula
\begin{equation}\label{eqn: Orb GV Multiple Cover}
\F^{\GW}_{\XX}(Q, \lambda, \bm{x}) = (-1)^{\epsilon} \sum_{\substack{
\beta \in H_{2}(X, \QQ) \\ \beta \neq 0}} \,\, \sum_{g \geq 0 } \,\,
\sum_{d \geq 1} \,\, \Psi_{d}\big( N^{\GW}_{\beta, g}(\bm{z})
\big)\big( 2 \sin \tfrac{d \lambda}{2} \big)^{2g-2} \frac{Q^{d
\beta}}{d}
\end{equation}
where
\[
N^{\GW}_{\beta, g}(\bm{z}) \coloneqq \sum_{\bm{m} \geq 0} \nrefined \,
\prod_{i=1}^{N} \sigma_{i}(\bm{z})^{m_{i}}
\]
and where $\epsilon = 1$ for case (i) and $\epsilon = 0$ for cases
(ii) and (iii). 
\end{theorem}

\begin{proof}
To prove this result, it will be necessary to convert the GW theory of
$\XX$ into the GW theory of $Y$ using the GW CRC and then convert to
the PT theory of $Y$ using the MNOP correspondence. The MNOP
correspondence was proven for compact CY3s by Pardon in 2023
\cite{Pardon}. We will argue below, on a case by case basis, that his
theory applies in the three non-compact cases in the theorem. The
correspondence will involve an overall minus sign in the local $K3$
case (which is an artifact of using localization to define the
invariants on the GW side and weighted Euler characteristics on the PT
side). The MNOP correspondence is the equality
\begin{equation}\label{eqn: MNOP correspondence for special local
cases Y}
\F^{\GW}_{Y}(Q,\lambda ,\bm{w}) = (-1)^{\epsilon} \log \ZZZ^{\PT}_{Y}(Q,y,\bm{w})
\end{equation}
after the change of variables $y=-e^{i\lambda}$. We defer the case
by case proof of Equation~\eqref{eqn: MNOP correspondence for special local
cases Y} and assume it for now.

From Equation~\eqref{eqn: MNOP correspondence for special local cases
Y} and Equation~\eqref{eqn: defn of BS partition function} it then
follows that

\begin{equation*}
\F^{\GW}_{Y, \nexc}(Q, \lambda, \bm{w}) = (-1)^{\epsilon} \log
\ZZZ^{\textsf{BS}}_{Y}(Q, y, \bm{w}).
\end{equation*}
Then applying the GW CRC and Theorem~\ref{thm: Main Weyl invariance
theorem} we get
\[
\F^{\GW}_{\XX}(Q, \lambda, \bm{x}) = (-1)^{\epsilon} \log \ZZZ^{\PT}_{\XX}(Q, y, \bm{w})
\]
where this equality involves the change of variables between $\bm{x}$
and $\bm{w}$ given by Equation~\eqref{eqn: CR change of
variables}. Applying Definition \ref{defn: Main GVPT Defn} we have
\[
\F^{\GW}_{\XX}(Q, \lambda, \bm{x}) = (-1)^{\epsilon} \sum_{\substack{
\beta \in H_{2}(X, \QQ) \\ \beta \neq 0}} \;  \sum_{\substack{g \geq 0\\d>0}} \,\, \Psi_{d}\big(N_{\beta, g}(\bm{w})\big) \big( 2
\sin \tfrac{d \lambda}{2} \big)^{2g-2} \, \frac{Q^{d \beta}}{d}
\]
where we have used the fact that
\[
\psi_{-(-y)^{d}} = \big( 2 \sin \tfrac{d \lambda}{2} \big)^{2}
\]
after setting $y=-e^{i\lambda}$. 

By Lemma \ref{lemma: AN change of vars Weyl-invariance}, under the
change of variables $w_{i} = z_{i}z_{i+1}^{-1}$, for all $i=1, \ldots,
N$, we have $\Phi_{i}(\bm{w}) = \sigma_{i}(\bm{z})$. Therefore,
\[
N^{\GW}_{\beta, g}(\bm{z}) = N_{\beta, g}(\bm{w})\big|_{w_{i} = z_{i} z_{i+1}^{-1}}
\]
which establishes the theorem, conditional on Equation~\eqref{eqn:
MNOP correspondence for special local cases Y}. This reduces the proof
to establishing the equality \eqref{eqn: MNOP correspondence for
special local cases Y}, which we do separately for the three cases.

\vskip2ex

\noindent \textbf{Case (i): Local orbifold $\bm{K3}$ Surface.}
In this case, Equation~\eqref{eqn: MNOP correspondence for special local
cases Y} is a theorem of Maulik-Thomas \cite[Theorem
6.3]{MaulikThomas2018} applied to
$Y=\overline{S}\times \CC$.

\vskip2ex
\noindent \textbf{Cases (ii) and (iii): Local surfaces with a
$\CC^{*}$-action and local curves}

For these two cases, we use the following observation. By work of
Pardon \cite{Pardon}, the standard MNOP correspondence for the local
geometry $Y$ holds if both the GW and PT theories are computed via
localization. Consequently, Equation~\eqref{eqn: MNOP correspondence
for special local cases Y} holds with $\epsilon =0$, when the
right-hand side, which is computed using Behrend-weighted Euler
characteristics, agrees with the similarly defined generating series
computed in PT theory by $\CC^{*}$-localization. By
Proposition~\ref{prop: C*-localization and weighted Euler char
definitions coincide when the action is CY} below, this holds whenever
the induced $\CC^{*}$-action on $\XX $ is a \emph{Calabi-Yau action},
namely that the induced action on $K_{\XX }$ is trivial (in such
cases, the induced action on $Y$ is also Calabi-Yau). By construction,
the $\CC^{*}$-action on $\XX$ is a Calabi-Yau action in cases (ii) and
(iii): indeed, for any vector bundle $\pi :E\to B$, one has an
equivariant isomorphism $K_{\Tot (E)}\cong \pi^{*}(K_{B}\otimes (\det
E)^{-1})$.  Both cases (ii) and (iii) are of this type with $\det E =
K_{B}$, and the given $\CC^{*}$-action has trivial determinant
character, so $K_{\XX }$ is equivariantly trivial.  Thus it suffices to
prove the following:

\begin{proposition}\label{prop: C*-localization and weighted Euler
char definitions coincide when the action is CY} Let $Y$ be a CY3 with
a Calabi-Yau $\CC^{*}$-action and let $P=P_{n}(Y,\beta )$ be the PT
moduli space. Assume that the fixed point locus $F=P^{\CC^{*}}\subset
P$ is compact. Then the PT invariants of $Y$ defined by virtual
localization and by weighted Euler characteristics coincide. Namely,
let $[F]^{\vir}$ be the induced virtual class on the fixed point
locus, let $N^{\vir}$ be the virtual normal bundle, let $\nu_{P}:P\to
\ZZ$ be Behrend's constructible function, and let $e(P,\nu_{P}) =
\sum_{k\in \ZZ} k\cdot e\left(\nu_{P}^{-1}(k) \right)$ be the
Behrend-function-weighted Euler characteristic. Then
\[
\int_{[F]^{\vir}} \frac{1}{e(N^{\vir})} = e(P,\nu_{P}) .
\] 
\end{proposition}
\smallskip

To prove this proposition, let
\[
\mathbb{E}_P^\bullet \longrightarrow \mathbb{L}_P
\]
be the standard $\CC^{*}$-equivariant perfect obstruction theory on
$P$ \cite{Pandharipande-Thomas}.
Since the action on $Y$ is Calabi--Yau, equivariant Serre duality
makes this obstruction theory equivariantly symmetric, with no
character twist:
\[
\mathbb{E}_P^\bullet
\cong
(\mathbb{E}_P^\bullet)^\vee[1].
\]
By \cite[Theorem~C]{LiQin2013}, the fixed locus $F$ inherits a
symmetric perfect obstruction theory from the fixed part of
$\mathbb{E}_P^\bullet$, and its Behrend function is related to that of
$P$ by
\[
\nu_P(p)
=
(-1)^{\dim T_p P-\dim T_p F}\nu_F(p),
\qquad p\in F.
\]

Write
\[
F=\coprod_\alpha F_\alpha
\]
as the disjoint union of its connected components, and let
$N_\alpha^{\vir}$ denote the restriction of $N^{\vir}$ to $F_\alpha$.
Let $N_{\alpha,+}^{\vir}$ be the sum of the strictly positive weight
summands of $N_\alpha^{\vir}$, regarded as a virtual bundle, and put
\[
r_\alpha=\operatorname{rk}N_{\alpha,+}^{\vir}.
\]
The equivariant symmetry of the obstruction theory identifies the
negative weight part with minus the dual of the positive weight part.
Thus, in equivariant $K$-theory,
\[
N_\alpha^{\vir}
=
N_{\alpha,+}^{\vir}
-
(N_{\alpha,+}^{\vir})^\vee.
\]
Since
\[
e\big((N_{\alpha,+}^{\vir})^\vee\big)
=
(-1)^{r_\alpha}e(N_{\alpha,+}^{\vir}),
\]
we obtain
\[
\frac{1}{e(N_\alpha^{\vir})}=(-1)^{r_\alpha}.
\]

We next identify this sign with the sign in the Li--Qin formula.  For
$p\in F_\alpha$, write the weight decomposition of the Zariski tangent
space as
\[
T_p P=\bigoplus_{j\in\ZZ}T_{p,j}P.
\]
Equivariant Serre duality identifies the weight-$j$ obstruction space
with $(T_{p,-j}P)^\vee$.  It follows that
\[
r_\alpha
=
\sum_{j>0}
\left(
\dim T_{p,j}P-\dim T_{p,-j}P
\right).
\]
Consequently,
\[
r_\alpha
\equiv
\sum_{j\ne 0}\dim T_{p,j}P
=
\dim T_p P-\dim T_p F
\pmod 2.
\]
In particular, the Li--Qin identity becomes
\[
\left.\nu_P\right|_{F_\alpha}
=
(-1)^{r_\alpha}\nu_{F_\alpha}.
\]

Since each $F_\alpha$ is proper and carries a symmetric obstruction
theory, Behrend's theorem \cite[Theorem~4.18]{KaiBehrend} gives
\[
\int_{[F_\alpha]^{\vir}}1
=
e(F_\alpha,\nu_{F_\alpha}).
\]
Therefore
\begin{align*}
\int_{[F]^{\vir}}\frac{1}{e(N^{\vir})}
&=
\sum_\alpha
(-1)^{r_\alpha}
\int_{[F_\alpha]^{\vir}}1 \\
&=
\sum_\alpha
(-1)^{r_\alpha}
e(F_\alpha,\nu_{F_\alpha}) \\
&=
e\big(F,\left.\nu_P\right|_F\big).
\end{align*}

Finally, the Behrend function $\nu_P$ is $\CC^{*}$-invariant.  Euler
characteristic localizes to the fixed locus for a $\CC^*$-action:
the complement of the fixed locus can be stratified by
$\CC^*$-orbits, all of which have Euler characteristic zero.  Hence,
for every $k\in\ZZ$,
\[
e\big(\nu_P^{-1}(k)\big)
=
e\big(\nu_P^{-1}(k)\cap F\big).
\]
It follows that
\[
e(P,\nu_P)
=
e\big(F,\left.\nu_P\right|_F\big).
\]
Combining the preceding identities proves
\[
\int_{[P^{\CC^*}]^{\vir}}
\frac{1}{e(N^{\vir})}
=
e(P,\nu_P).
\]
This proves Proposition~\ref{prop: C*-localization and weighted Euler
char definitions coincide when the action is CY} which then completes
the proof of Theorem~\ref{thm: Orb GV Multiple Cover}.
\end{proof}

\begin{corollary}\label{cor: Main GW corollary}
Assume $\XX$ is as in Theorem \ref{thm: Orb GV Multiple Cover}. The
total GV invariants $n_{g}^{\XX}(\beta)$ are encoded into the total GW
potential of $\XX$ by the formula:
\[
\F^{\GW}_{\XX}(Q,\lambda ,\bm{x})\Big\rvert_{x_{j} = \frac{\pi}{N+2}\csc
\left(\frac{j\pi}{N+1} \right)} = (-1)^{\epsilon} \sum_{\beta \neq 0}\, \sum_{g \geq 0}\, \sum_{d
\geq 1}\, n^{\XX}_{g}(\beta) \, \left( 2 \sin \tfrac{d \lambda}{2} \right)^{2g-2}\,
\frac{Q^{d \beta}}{d}.
\]
\end{corollary}
\begin{proof}
The quantities $N^{\GW}_{\beta, g}(\bm{z})$ determine
$\F^{\GW}_{\XX}(Q,\lambda ,\bm{x})$ as in Theorem \ref{thm: Orb GV
Multiple Cover}, which involves the change of variables (\ref{eqn: CR
change of variables}) between $\bm{x}$ and $\bm{z}$. It is clear that
the total GV invariants $n_{g}^{\XX}(\beta)$ arise by making a
suitable specialization of the $\bm{z}$-variables in $N^{\GW}_{\beta,
g}(\bm{z})$ such that $\sigma_{i}(\bm{z}) = 1$ for all $i = 1, \ldots,
N$. By the proof of Corollary \ref{cor: unrefined PT/GV
specialization 2}, this is achieved by setting 
\[
z_{k} = - \zeta^{-k}
\]
for all $k=1,
\ldots, N+1$ where 
\[
\zeta =\exp\left(\tfrac{2\pi i}{N+2} \right).
\]
By way of the change of variables (\ref{eqn: CR change
of variables}), we want to show that this specialization arises by
setting
\[
x_{j} = \frac{\pi}{N+2} \csc\bigg(\frac{ j \pi }{N+1} \bigg) = \frac{
2 \pi i}{N+2} \, \frac{\omega^{\frac{j}{2}}}{\omega^{j}-1}
\]
where 
\[
\omega =\exp\left(\tfrac{2\pi i}{N+1} \right).
\]
Substituting this expression into (\ref{eqn: CR change of variables}) and simplifying, we have
\[
z_{k} = -\exp \bigg( - \frac{2 \pi i}{(N+1)(N+2)} \, \bigg( (k-\tfrac{1}{2})(N+2) + \sum_{j=1}^{N} \frac{\omega^{jk}}{\omega^{j} - 1}\bigg) \bigg).
\]
It is straightforward to verify the identity
\[
\sum_{j=1}^{N} \frac{\omega^{jk}}{\omega^{j} - 1} = \frac{N}{2}-k+1
\]
from which the result follows. 
\end{proof}

For a primitive class on an ordinary Calabi-Yau threefold, the genus
$0$ GV invariant is simply the genus $0$ GW invariant because there
are no multiple covers, and no collapsing components by stability. But
in the orbifold setting, we can have collapsing components mapping to
orbifold points. However, it is a consequence of Corollary~\ref{cor:
Main GW corollary} that we have
\begin{corollary}
Let $\beta$ be a primitive class on $\XX$, a local orbifold CY3 of
$A_{N}$-type. Then
\begin{equation}\label{eqn: genus 0 primitive invariants}
n^{\XX}_{0}(\beta) = \F_{\XX, \beta}^{0}(\bm{x}) \big|_{x_{j} = \tfrac{\pi}{N+2} \csc\big(\tfrac{j \pi}{N+1}\big)}.
\end{equation}
\end{corollary}

\begin{example}
Let $\XX$ be a local orbifold CY3 of $A_{1}$-type, and suppose $C
\subset \XX$ is a rational curve representing a reduced and
irreducible class $\beta$. If $C$ is rigid and intersects the orbifold
locus transversely in $p$ points, then the contribution of $C$ to the
genus zero orbifold GW potential can be computed using a result of
Wise \cite{Wise-hyperelliptic-hodge}, and organizing the results into
a generating series easily yields
\[
\sum_{m \geq 0} \langle \gamma^{m} \rangle_{0, \beta}^{\XX}
\frac{x^{m}}{m!} = \big( 2\sin(\tfrac{x}{2}) \big)^{p}.
\]
By Equation~\eqref{eqn: genus 0 primitive invariants}, the
contribution of $C$ to the GV invariant $n^{\XX}_{0}(\beta)$ arises by
setting $x = \tfrac{\pi}{3} \csc(\tfrac{\pi}{2}) =
\tfrac{\pi}{3}$. Making this substitution into the above generating
series, we see
\[
\big( 2\sin( \tfrac{\pi}{6}) \big)^{p} = 1.
\]
Thus, $C$ contributes $1$ to the invariant $n^{\XX}_{0}(\beta )$.
\end{example}

\subsection{The Local Teardrop}\label{subsection: local teardrop}

Let $\PP^{1}(N+1, 1)$ denote the $\PP^{1}$-orbifold with
torus-invariant points $p_{0}$ and $p_{\infty}$, with
$B\mu_{\scriptscriptstyle N+1}$ orbifold structure at
$p_{0}$. Consider the two orbifold line bundles $\OO(-p_{0})$ and
$\OO(-p_{\infty})$ characterized by
\[
\OO(-p_{0})^{\otimes (N+1)} \isom c^{*} \OO(-1), \quad \quad \quad \quad \OO(-p_{\infty}) \isom c^{*} \OO(-1)
\]
where $c : \PP^{1}(N+1, 1) \to \PP^{1}$ is the map to the coarse space. The \emph{local teardrop} is defined to be the following total space of a rank two bundle
\[
\XX = \Tot\big( \OO(-p_{0}) \oplus \OO(-p_{\infty}) \big)
\]
which is an example of a toric orbifold CY3 of $A_{N}$-type. Observing that the fiber of the rank two bundle over $p_{0}$ is isomorphic to
\[
\big[ \CC/ \mu _{N+1} \big] \times \CC
\]
we see that there is a single non-compact curve in $\XX$ carrying transverse orbifold structure. 

In \cite{Johnson-Pandharipande-Tseng-localPab}, the GW theory of local $\PP^{1}$-orbifolds is computed. In the case of the local teardrop, \cite[Eqn. 21]{Johnson-Pandharipande-Tseng-localPab} takes the form
\begin{equation}\label{eqn: Pandharipande-Johnson-Tseng}
\begin{split}
\sum_{g \geq 0} & \langle \gamma_{1}^{m_{1}} \cdots \gamma_{N}^{m_{N}} \rangle_{g, d[\PP^{1}]}^{\XX} \, \lambda^{2g-2} \\
& = (-1)^{-\frac{1}{N+1}( \sum_{k=1}^{N} k m_{k} - d) + d + \sum_{k=1}^{N} m_{k}} \bigg(\frac{d}{N+1}\bigg)^{\sum_{k=1}^{N} k m_{k} - 1} \big( 2 \sin \tfrac{d \lambda}{2} \big)^{-2}
\end{split}
\end{equation}
subject to the constraint $\sum_{k=1}^{N} k m_{k} \equiv d \pmod N+1$ imposed by requiring the twisted stable map be representable.

\begin{proof}[Proof of Proposition \ref{prop: Multiple Cover Formula for Local Teardrop}]
It is straightforward to assemble (\ref{eqn: Pandharipande-Johnson-Tseng}) into the GW potential, which takes the form
\[
\F_{\XX}^{\GW}(Q, \lambda, \bm{x}) = \sum_{d \geq 1} [ A_{d}(\bm{x}) ]^{\sharp_{d}} \big( 2 \sin \tfrac{d \lambda}{2} \big)^{-2} \frac{Q^{d}}{d}
\]
where 
\[
A_{d}(\bm{x}) = (N+1) (-1)^{d} \omega^{\frac{d}{2}} \sum_{ \bm{m} \geq 0} \,\, \prod_{k=1}^{N} \frac{1}{m_{k}!} \big(-\omega^{-\frac{k}{2}} \tfrac{d x_{k}}{N+1}\big)^{m_{k}}, \,\,\,\,\,\,\,\, \omega = \exp\big(\tfrac{2 \pi i}{N+1}\big)
\]
and $[\cdots]^{\sharp_{d}}$ means keep only monomials
$x_{1}^{m_{1}}\cdots x_{N}^{m_{N}}$ of total weighted degree congruent
to $d$ modulo $N+1$ with $\dg x_{k} = k$.

Straightforward manipulations of $A_{d}(\bm{x})$ result in
\begin{equation*}
\begin{split}
A_{d}(\bm{x}) & = (N+1) (-1)^{d} \omega^{\frac{d}{2}} \prod_{k=1}^{N} \,\, \sum_{m \geq 0} \frac{1}{m!}  \big(-\omega^{-\frac{k}{2}} \tfrac{d x_{k}}{N+1}\big)^{m}\\  
& = (N+1) (-1)^{d} \omega^{\frac{d}{2}} \prod_{k=1}^{N} \exp \big(-\omega^{-\frac{k}{2}} \tfrac{d x_{k}}{N+1}\big) \\
& = (N+1) (-1)^{d} \omega^{\frac{d}{2}} \exp \bigg( - \frac{d}{N+1}
\sum_{k=1}^{N} \omega^{-\frac{k}{2}} x_{k}\bigg)
\end{split}
\end{equation*}
In order to extract the terms of $A_{d}(\bm{x})$ of total weighted
degree $d$ modulo  $N+1$ as described above, we employ the following
observation
\[
[A_{d}(\bm{x})]^{\sharp_{d}} = \frac{1}{N+1} \sum_{j=1}^{N+1} \omega^{-jd} A_{d}(\omega^{j}x_{1}, \ldots, \omega^{jk}x_{k}, \ldots, \omega^{jN}x_{N}).
\]
This identity holds for any power series in $N$ variables, and does not rely on the specific form of $A_{d}(\bm{x})$. Plugging in the above expression for $A_{d}(\bm{x})$, we conclude
\begin{equation*}
\begin{split}
[A_{d}(\bm{x})]^{\sharp_{d}} = & \sum_{j=1}^{N+1} \bigg(-\omega^{\frac{1}{2}-j} \exp\bigg(-\frac{1}{N+1} \sum_{k=1}^{N} \omega^{k(j-\frac{1}{2})} x_{k}\bigg) \bigg)^{d} \\
& = \sigma_{1}(z_{1}^{d}, \ldots, z_{N+1}^{d}) = \sigma_{1}(\bm{z}^{d})
\end{split}
\end{equation*}
with the penultimate equality following from the change of variables (\ref{eqn: CR change of variables}). 
\end{proof}


\bibliographystyle{amsplain}
\bibliography{localbib}

\end{document}

%% file: notation-table.tex

\begingroup
\footnotesize
\setlength{\intextsep}{4pt}
\setlength{\abovecaptionskip}{2pt}
\setlength{\belowcaptionskip}{2pt}
\setlength{\tabcolsep}{3pt}
\renewcommand{\arraystretch}{1.00}

\begin{table}[H]
\centering
\caption{Guide to notation.}
\label{tab:notation}
\begin{tabular}{@{}>{\raggedright\arraybackslash}p{.26\textwidth}
                    >{\raggedright\arraybackslash}p{.485\textwidth}
                    >{\raggedright\arraybackslash}p{.20\textwidth}@{}}
\hline
\textbf{Symbol} & \textbf{Meaning} & \textbf{Location} \\
\hline
\multicolumn{3}{@{}l@{}}{\textbf{Geometry and numerical groups}} \\
\hline
$\XX$, $X$, $\pi:Y\to X$
& The local orbifold CY3, its coarse space, and the crepant resolution.
& Def.~\ref{defn: new local orbifold CY3 defn}; \S\ref{subsection: Derived McKay and Ktheory}. \\
$\BB$, $B$, $D_i$, $C_i$
& The orbifold locus, its coarse curve, an exceptional divisor, and its generic fiber.
& Def.~\ref{defn: new local orbifold CY3 defn}; Lem.~\ref{lemma: Euler pairing lemma}. \\
\mbox{$N(Y),\ N_{\leq1}(Y),\ N_0(Y)$}
& Compactly supported numerical $K$-theory of $Y$ and its subgroups in dimensions at most one and zero.
& \S\ref{subsection: Derived McKay and Ktheory}. \\
\mbox{$N(\XX),\ N_{\leq1}(\XX),\ N_0(\XX)$}
& The corresponding numerical groups for the orbifold $\XX$.
& \S\ref{subsection: preliminaries on orbifolds and PT theory}. \\
$\FM:N(Y)\xrightarrow{\sim}N(\XX)$
& The numerical Fourier--Mukai isomorphism induced by derived McKay.
& Eq.~\eqref{eqn: FM on num K-theory}. \\
$N_{\exc}(Y)\cong N_0(\XX)\cong\ZZ\oplus\Lambda_{A_N}$
& Exceptional classes on $Y$ identified with orbifold point classes; the second summand is the root lattice.
& Lem.~\ref{lemma: FM with N_exc and N_0}. \\
$N_1(Y)^{\natural}\cong N_1(\XX)$
& Non-exceptional curve classes, identified by $\FM$.
& Defs.~\ref{defn: non-exceptional curve class defn} and~\ref{defn: definition of N_1 on the orbifold}. \\
$N_{\leq1}(Y)$, $N_{\leq1}(\XX)$
& Integral classes inside the ambient groups with weight-lattice component $\LambdaWeightAN$.
& Lems.~\ref{lemma: Ktheory on resolution} and~\ref{lemma: basic K-theory lemma}. \\
$H_2(Y)^{\natural}$
& The cycle-class image of $N_1(Y)^{\natural}$; equivalently, the subspace orthogonal to the $D_i$.
& Def.~\ref{defn: H2(Y)perp}. \\
\hline
\multicolumn{3}{@{}l@{}}{\textbf{Root data and variables}} \\
\hline
$\Lambda_{A_N}\subset\LambdaWeightAN$, $R^+$, $W$
& The root and weight lattices, positive roots, and Weyl group of type $A_N$.
& Def.~\ref{defn: defn of AN root lattice}; \S\ref{sec: root theory section}. \\
$s_1,\ldots,s_N$
& The simple reflections generating $W$, acting on the variables $\bm w$.
& Eq.~\eqref{eqn: simple root reflections}. \\
$\alpha_i$, $\omega_i$
& The simple roots and fundamental weights of the $A_N$ root system.
& Eq.~\eqref{eqn: AN simple roots}; \S\ref{sec: root theory section}. \\
$v=\sum_i v_i\alpha_i$, $\bm w^v=\prod_iw_i^{v_i}$
& Simple-root coordinates of a weight and the corresponding group-ring monomial.
& Def.~\ref{defn: group ring definition}. \\
$\Phi_i(\bm w)$
& The orbit sum of $\omega_i$; these generate $\ZZ[\LambdaWeightAN]^W$.
& Thm.~\ref{thm: crucial Fulton-Harris thm on Weyl invariants}. \\
$\bm z=(z_1,\ldots,z_{N+1})$, $\sigma_i(\bm z)$
& Variables with $\prod_kz_k=1$ and $w_i=z_i z_{i+1}^{-1}$; after substitution, $\Phi_i(\bm w)=\sigma_i(\bm z)$.
& Lem.~\ref{lemma: AN change of vars Weyl-invariance}. \\
$Q^\beta$, $y^n$, $\bm w^v$
& Variables recording the non-exceptional curve class, Euler characteristic, and weight component.
& Eq.~\eqref{eqn: PT series defn}. \\
$\lambda$, $\bm x=(x_1,\ldots,x_N)$
& The genus parameter and twisted-sector variables in orbifold GW theory.
& Def.~\ref{defn: full orbi GW potential}. \\
$\omega$, $\zeta$
& Roots of unity used in the crepant-resolution substitution and the total-GV specialization.
& Eq.~\eqref{eqn: CR change of variables}; Cor.~\ref{cor: unrefined PT/GV specialization 2}. \\
\hline
\multicolumn{3}{@{}l@{}}{\textbf{Enumerative invariants and series}} \\
\hline
$\PT_{\beta,n,v}(\XX)$, $\ZZZ_{\XX}^{\PT}$; $\PT_{\beta+v,n}(Y)$, $\ZZZ_Y^{\PT}$
& PT invariants and partition functions of the orbifold and its crepant resolution.
& Eqs.~\eqref{eqn: PT partition function on local orbifold 2} and~\eqref{eqn: PT series defn}. \\
$M(x,y)$, $\ZZZ_{\exc}$
& The weighted MacMahon function and the PT series of purely exceptional classes.
& Lem.~\ref{lem: formula for Zexc}. \\
$\ZZZ_Y^{\BS}$, $\ZZZ_\beta^{\BS}$
& The Bryan--Steinberg series and its $Q^\beta$-coefficient.
& Eq.~\eqref{eqn: defn of BS partition function}. \\
$\ZZZ_\delta^{\BM,D}$, $g_\delta^D$
& The Buelles--Moreira series for a ruled divisor $D$ and its rational form.
& Proof of Thm.~\ref{thm: Main Weyl invariance theorem}, part (B). \\
$\ZZZ_\beta^{\PT}$, $f_\beta(y,\bm w)$
& The $Q^\beta$-coefficient of the orbifold PT series on $\XX$ and its Weyl-invariant rational form.
& Thm.~\ref{thm: Main Weyl invariance theorem}. \\
$n_g^Y(\gamma)$, $\Psi_d$, $\psi_x$
& Ordinary GV invariants, the Adams substitution, and the factor $\psi_x=(x^{1/2}+x^{-1/2})^2$.
& Eq.~\eqref{eqn: Ordinary GV/PT formula}; Def.~\ref{defn: Main GVPT Defn}. \\
$n_g^{\XX}(\beta;\bm m)$, $N_{\beta,g}(\bm w)$
& The graded orbifold GV invariants and their Weyl-invariant generating polynomial.
& Def.~\ref{defn: Main GVPT Defn}. \\
$n_g^{\XX}(\beta)$
& The sum of the graded invariants over $\bm m$.
& Eq.~\eqref{eqn: total GV invariants}. \\
$\langle\cdots\rangle_{g,\beta}^{\XX}$, $\F_{\XX}^{\GW}$; $\F_Y^{\GW}$, $\F_{Y,\nexc}^{\GW}$
& Localized orbifold GW invariants and the full orbifold, resolution, and non-exceptional resolution potentials.
& Defs.~\ref{defn: full orbi GW potential} and~\ref{defn: GW potentials F and Fn-exc}. \\
$N_{\beta,g}^{\GW}(\bm z)$
& The graded GV polynomial in the $\bm z$ variables used in the orbifold GW/GV formula.
& Thm.~\ref{thm: Orb GV Multiple Cover}. \\
\hline
\end{tabular}
\end{table}

\endgroup